\documentclass[12pt]{article}
\usepackage[margin=.88in]{geometry}

\usepackage{multirow}
\usepackage{amsmath, amssymb,amsthm}
\usepackage{bm}
\usepackage[pdftex]{graphicx}
\usepackage{comment}
\usepackage{caption}
\usepackage{subcaption}
\usepackage{enumitem}
\usepackage{natbib}
\usepackage{diagbox}
\usepackage{lineno}
\usepackage[normalem]{ulem}

\usepackage{hyperref}
\hypersetup{
	colorlinks=true,
	linkcolor=blue,
	filecolor=blue,
	urlcolor=red,
	citecolor=blue,
}

\DeclareMathAlphabet\mathbfcal{OMS}{cmsy}{b}{n}

\newcommand{\be}{\begin{equation}}
\newcommand{\ee}{\end{equation}}
\newcommand{\bs}{\begin{split}}
\newcommand{\es}{\end{split}}
\newcommand{\bea}{\begin{eqnarray}}
\newcommand{\eea}{\end{eqnarray}}
\newcommand{\beas}{\begin{eqnarray*}}
\newcommand{\eeas}{\end{eqnarray*}}

\usepackage{color}

\newcommand{\E}{\mathbb{E}}

\newcommand{\R}{\mathbb{R}}

\newcommand{\bcM}{{\mathbfcal{M}}}

\newcommand{\bcS}{{\mathbfcal{S}}}

\newcommand{\bcT}{{\mathbfcal{T}}}

\newcommand{\bcK}{{\mathbfcal{K}}}

\newcommand{\polylog}{\operatorname{polylog}}

\newtheorem{proposition}{Proposition}
\newtheorem{theorem}{Theorem}
\newtheorem{conjecture}{Conjecture}
\newtheorem{lemma}{Lemma}
\newtheorem{corollary}{Corollary}
\newtheorem{definition}{Definition}
\newtheorem{remark}{Remark}
\newtheorem{assumption}{Assumption}

\usepackage{setspace}
\begin{document}

\title{From Reverse Detection--Estimation Gaps to Computational Lower Bounds for Norm Approximation}

\author{Runshi Tang\footnote{Department of Statistics, Columbia University},\quad
Yuefeng Han\footnote{Department of Applied and Computational Mathematics and Statistics, University of Notre Dame},\quad and\quad
Anru R. Zhang\footnote{Department of Biostatistics \& Bioinformatics and Department of Computer Science, Duke University}}

\date{}
\maketitle

\begin{abstract}
We develop a general reduction scheme that converts reverse detection--estimation gaps into computational lower bounds for norm approximation. If an efficiently computable estimator has error well below a conjectured computational detection threshold, then a sufficiently accurate norm approximator applied to this estimator would yield an efficient test, transferring average-case detection hardness to worst-case approximation hardness. The ratio between the detection and estimation scales becomes a lower bound on the achievable approximation factor.

We apply the reduction in three settings. Using a previously established reverse gap for high-order cumulant tensors, we obtain a conditional lower bound for polynomial-time approximation of the order-$d$ tensor spectral norm. We then establish new reverse gaps for the $k$-sparse matrix spectral norm and a normalized $k$-sparse matrix cut norm under planted-clique hardness. The resulting lower bounds are of order $\sqrt{k}$ up to a logarithmic factor, matching the simple deterministic $\sqrt{k}$ certificates, whereas previous hardness results for the sparse principal component objective rule out only constant factors. Finally, under a model-specific low-degree conjecture, a new reverse gap forces polynomially growing approximation factors for the order-$d$ tensor cut norm when $d\ge3$. These results identify reverse detection--estimation gaps as a systematic source of computational lower bounds for norm approximation.
\end{abstract}

\section{Introduction}\label{sec:intro}

%

{\it Norm approximation} is computationally straightforward for many classical matrix norms: exact evaluation is possible in polynomial time. Under sparsity constraints or higher-order tensor structure, however, even dimension-dependent approximation can become computationally difficult.
Let $\|\cdot\|$ be a norm on a finite-dimensional space $\mathcal F_p$, where $p$ indexes the problem size,
and let $f_p:\mathcal F_p\to\mathbb R_+$ satisfy
\[
\rho_{f,p}^{-1}\|T\|\le f_p(T)\le \zeta_{f,p}\|T\|,
\qquad T\in\mathcal F_p,
\]
for approximation factors $\rho_{f,p},\zeta_{f,p}\ge 1$.
We call $\gamma_{f,p}:=\rho_{f,p}\zeta_{f,p}$ the distortion of $f_p$.
The basic question studied in this paper is:
how accurately can such norms be approximated, or more specifically, how small can $\gamma_{f,p}$ be, in polynomial time as $p$ grows?

A seemingly different but closely related phenomenon is the {\it reverse detection--estimation gap}, in which the classical ordering of the two tasks, detection being no harder than estimation, is reversed.
Consider a testing problem
\[
H_0:\mu_p=0
\qquad\text{versus}\qquad
H_1:\|\mu_p\|\ge s_p,
\]
and suppose that there exists a polynomial-time computable estimator $\widehat{\mu}_p$ such that
$\|\widehat{\mu}_p-\mu_p\|\le r_p$ with high probability under both hypotheses.
If $s_p/r_p\to\infty$, so that estimation is possible at a scale much smaller than the signal separation, while detection at scale $s_p$ remains computationally hard, we say that the problem exhibits a \emph{reverse detection--estimation gap}, with gap size $s_p/r_p$.
Such a gap suggests a direct route from statistical detection hardness to computational hardness of approximating the target norm.
Indeed, if $\|\widehat{\mu}_p\|$ could be evaluated efficiently, then thresholding it between the null scale $r_p$ and the alternative scale $s_p-r_p$ would yield an efficient test.
More generally, a sufficiently accurate approximation to $\|\widehat{\mu}_p\|$ can yield the same contradiction, so the ratio $s_p/r_p$ naturally determines an approximation-hardness scale.

\subsection{Main contributions}

In this paper, we make two main contributions.
First, we develop a general reduction that converts a reverse detection--estimation gap into a computational lower bound for norm approximation.
Roughly, a reverse gap of size $s_p/r_p$ rules out polynomial-time norm approximation with distortion $o(s_p/r_p)$ under the corresponding detection-hardness assumption.
Second, we instantiate the reduction in three settings, two of which require establishing new reverse detection--estimation gaps. Using a previously established reverse gap for high-order cumulant tensors, we obtain a conditional lower bound, polynomial in the dimension $p$, for polynomial-time approximation of the order-$d$ tensor spectral norm. We then establish new reverse gaps that yield conditional lower bounds of order $\sqrt{k/\log(ep/k)}$ for the $k$-sparse matrix spectral norm and a normalized $k$-sparse matrix cut norm, and polynomially growing lower bounds for the order-$d$ tensor cut norm with $d\ge3$.
The resulting approximation lower bounds are conditional: the sparse-matrix results rely on planted-clique detection hardness, whereas the tensor results rely on model-specific low-degree conjectures.
We describe these results in more detail below.

\paragraph{A general reduction from reverse gaps to approximation hardness.}
We first formalize the preceding connection for a general normed parameter space. Suppose that strong detection is computationally hard at signal scale $s_p$, while a polynomial-time estimator achieves error $r_p$ in the target norm under both hypotheses. We show that a reverse gap of size $s_p/r_p$ rules out norm approximation with distortion $o(s_p/r_p)$. More precisely, under the uniformity, factor-computability, and threshold-closure conditions specified in Section~\ref{sec:framework}, every admissible polynomial-time norm approximator satisfies
\[
\limsup_{p\to\infty}\frac{\gamma_{f,p}r_p}{s_p}\ge \frac14.
\]
Thus, whenever $s_p/r_p\to\infty$, a reverse detection--estimation gap rules out approximation distortion $o(s_p/r_p)$; the constant $1/4$ is inessential to this scaling consequence.
The main role of the general theorem is to isolate the precise conditions under which this implication is algorithmically valid. These conditions hold for standard certificate-producing algorithms in the literature, whose approximation factors, such as $\sqrt{k}$ or $p^{(d-2)/2}$, are explicit and efficiently computable.
Conceptually, the reduction converts average-case detection hardness in a statistical model into a worst-case lower bound for norm approximation.

\paragraph{Tensor spectral norm: an established reverse gap.}
Our first application uses a reverse detection--estimation gap for tensor spectral norm that was previously established in the high-order cumulant model \citep{tang2026detectionharderestimationcertain}.
Let $p$ denote the ambient dimension and let $d\ge3$ denote the tensor order.
In that model, the low-degree likelihood-ratio norm remains bounded at a signal scale for which the empirical cumulant tensor can already be estimated with substantially smaller error in the tensor spectral norm.
Equivalently, the established statistical result provides a reverse gap of polynomial order in $p$, up to logarithmic factors, which our reduction transfers directly to approximation hardness.
Applying the general reduction gives, under the corresponding model-specific low-degree conjecture,
\[
\limsup_{p\to\infty}
\frac{\gamma_{f,p}
(\log p)^{9d/2}}{p^{d/4-1/2}}
>
0
\]
for every uniform deterministic polynomial-time approximator of the symmetric tensor spectral norm whose approximation factors are efficiently computable.
Thus, a reverse gap originally arising from a statistical testing problem directly yields a worst-case computational lower bound for tensor norm approximation.

\paragraph{Sparse matrix norms: new reverse gaps.}
We next establish new reverse detection--estimation gaps for the $k$-sparse spectral norm and a normalized $k$-sparse cut norm on $p$-by-$p$ matrices, defined in \eqref{eq_sparse_spectral_norm} and \eqref{eq_sparse_cut_norm}, respectively. 
In the planted-clique model, the centered adjacency matrix is an efficiently computable estimator of the planted mean matrix.
In both the $k$-sparse spectral norm and the normalized $k$-sparse matrix cut norm, the planted signal has size of order $k$, while the stochastic estimation radius is $\sqrt{k\log(ep/k)}$.
This gives the reverse-gap ratio $\sqrt{k/\log(ep/k)}$.
Under randomized planted-clique detection hardness, we therefore rule out distortion
$o\left(
\sqrt{k/\log(ep/k)}
\right)$ for randomized global approximators satisfying the uniformity, factor-computability, and threshold-closure conditions of the theorems.
For the sparse spectral norm, this has a direct connection to sparse principal component analysis: on positive semidefinite inputs, $\|M\|_{\mathrm{sp},k}$ (defined in Section~\ref{sec:sparse}) is exactly the sparse principal component objective. 
Whereas previously available hardness results for this objective rule out only fixed-factor approximations, our result yields a dimension-dependent lower bound.
Combined with the deterministic $\sqrt{k}$ certificates in Remark~\ref{remark:sparse_upper_gap},
the lower bounds for both sparse matrix norms are within a factor of order $\sqrt{\log(ep/k)}$ of these simple deterministic certificates.

\paragraph{Tensor cut norm: a new reverse gap.}
We further construct a reverse detection--estimation gap for the normalized cut norm for tensors in $\R^{p^d}$.
Under an asymmetric spiked Gaussian tensor model \eqref{eq_tensor_cut_model}, we obtain a reverse gap of order
\[
\frac{p^{d/4-1/2}}{(\log p)^B},
\qquad
B>\frac{d-2}{2}.
\]
Under the corresponding model-specific low-degree conjecture, this implies
\[
\limsup_{p\to\infty}
\gamma_{f,p}\frac{(\log p)^B}{p^{d/4-1/2}}
>0
\]
for every uniform polynomial-time approximator of the order-$d$ tensor cut norm whose approximation factors are efficiently computable.

\vspace{\baselineskip}

Taken together, these results show that reverse detection--estimation gaps provide a systematic route to norm-approximation hardness. Rather than constructing a separate direct reduction for each approximation problem, one can identify a statistical model in which estimation in the target norm is efficient while detection remains computationally hard, and then transfer the resulting separation of scales into a worst-case approximation lower bound.

\subsection{Related literature}

\paragraph{Statistical--computational gaps and average-case reductions.}
A large literature studies gaps between statistically attainable performance and what can be achieved by computationally efficient procedures. General perspectives based on convex relaxation, average-case complexity, and reductions from planted problems have been developed in a variety of settings \citep{chandrasekaranJordan2013,berthetRigolletComplexity2013,brennanBreslerHuleihel2018,brennanBreslerSecretLeakage2020}. Concrete examples include sparse principal component analysis, submatrix detection and localization, sparse canonical correlation analysis, sparse matrix detection, graphon estimation, and other structured high-dimensional problems \citep{maWuSubmatrix2015,wangBerthetSamworthSparsePCA2016,gaoMaZhouSparseCCA2017,caiLiangRakhlinSubmatrix2017,caiWuSparseMatrix2020,luoGaoGraphon2024,tanKlusowskiBalasubramanian2026}. More refined reduction frameworks have sought to transfer hardness between statistically distinct planted problems, including reductions tailored to sparse PCA and reductions based on variants of planted clique and hypergraphic planted clique \citep{brennanBreslerSparsePCA2019,brennanBreslerHuleihel2018,brennanBreslerSecretLeakage2020,luoZhangHypergraphicClique2020}. Related work has also studied the precise error attainable by computationally efficient tests, rather than only the location of a computational threshold \citep{moitraWeinTesting2025}. The reverse detection--estimation phenomenon used here was established for high-order moment and cumulant tensors in \cite{tang2026detectionharderestimationcertain}.

\paragraph{Low-degree polynomials, statistical queries, and sum-of-squares.}
Low-degree polynomials and the sum-of-squares hierarchy have become important tools for predicting and proving computational barriers in planted statistical problems \citep{hopkins2018statistical,kunisky2019notes,hopkins2017power,schramm2022computational}. Their behavior has been analyzed for planted clique, sparse PCA, tensor problems, constrained principal component problems, and related certification tasks \citep{mekaPotechinWigderson2015,maWigdersonSparsePCA2015,bandeiraKuniskyWein2020,potechinRajendranSparsePCA2022}. Connections between low-degree tests and other restricted computational models, notably statistical-query algorithms, provide additional justification for using low-degree calculations as algorithmic evidence \citep{brennanBreslerHopkinsLiSchramm2021}. At the same time, the scope of the low-degree heuristic requires care. Counterexamples to broad formulations have been constructed \citep{holmgrenLowDegree2020}, and recent work gives counterexamples to stronger quasi-polynomial-time and polynomial-time formulations \citep{buhaiHsiehJainKothari2025,maoLowDegreeFalse2026}. In a complementary direction, recent results establish rigorous algorithmic consequences of low-degree bounds under additional conditions \citep{hsiehKaneKothariLiMohantyTiegel2026}. This evolving picture is one reason that the tensor results in this paper formulate the low-degree-to-polynomial-time step explicitly as a model-specific conjecture, while keeping the low-degree likelihood-ratio calculations themselves separate from that conjectural step.

\paragraph{Tensor norms and multilinear optimization.}
Tensor spectral and nuclear norms are fundamental objects in multilinear algebra, rank-one approximation, polynomial optimization, and tensor decomposition \citep{limSingularValuesEigenvalues2005,koldaTensorDecompositionsApplications2009,deLathauwerDeMoorVandewalle2000}. In contrast with matrix operator and nuclear norms, many basic tensor-norm problems become computationally intractable for tensors of order three and higher: exact computation and several approximation tasks are NP-hard, with connections to multilinear optimization and quantum-information problems \citep{hillarMostTensorProblems2013a,harrowTestingProductStates2013,friedlandNuclearNormHigherOrder2018}. A substantial algorithmic literature develops relaxations and approximation procedures based on semidefinite programming, tensor flattenings and partitions, sphere coverings, and sum-of-squares methods \citep{nieWangRankOneTensor2014,jiangMaZhangTensorPCA2015,brandaoEstimatingOperatorNorms2015,bhattiproluWeakDecoupling2017,bhattiproluSumofSquaresCertificates2017,liBoundsSpectralNuclear2016,chenLiTensorSpectralPNorm2020,heApproximatingTensorNorms2023}. The geometry and computation of tensor spectral and nuclear norms, including their duality and their relations to matrix flattenings, have also been studied directly \citep{friedlandComputationalComplexityDuality2016,huTensorFlattenings2015,friedlandWangTensorSpectral2020}. More recent work gives sharper algorithms in special dimensional regimes and develops further approximation results for tensor nuclear and spectral $p$-norms \citep{huJiangLiOrderThree2025,guanHeJiangLiNuclearPNorm2025}.

\paragraph{Tensor PCA and statistical tensor models.}
There is a parallel literature on statistical and computational thresholds for random tensor models. Spiked tensor PCA provides a canonical example in which information-theoretic, spectral, and computational thresholds can differ substantially \citep{montanariRichardTensorPCA2014,montanariReichmanZeitouni2015,hopkinsTensorPCA2015}. Sum-of-squares, statistical-query, communication-complexity, and related hierarchies have been used to characterize or predict the power of broad classes of algorithms in tensor PCA and its variants \citep{hopkins2017power,dudejaHsuTensorPCA2021,chooDorsiSparseTensorPCA2021,dudejaHsuCommunicationTensorPCA2024,weinElAlaouiMooreKikuchi2025}. Related work studies nonnegative PCA and higher-order tensor singular-vector estimation from statistical and computational viewpoints \citep{montanariRichardNonnegativePCA2016,zhangXiaTensorSVD2018}.

\paragraph{Sparse PCA and restricted isometry.}
The sparse spectral norm is closely connected to a long line of work on sparse principal component analysis. Early formulations and algorithms include exact or greedy combinatorial methods, $\ell_1$-based formulations, semidefinite relaxations, generalized power methods, iterative thresholding, and low-rank approximation techniques \citep{moghaddamWeissAvidanSparsePCA2005,zouHastieTibshiraniSparsePCA2006,dAspremontSparsePCA2007,aminiWainwrightSparsePCA2009,journeeNesterovRichtarikSepulchre2010,maSparsePCAIterative2013,papailiopoulosDimakisKorokythakis2013}. Statistical theory subsequently characterized optimal estimation rates and information-theoretic thresholds, while also revealing regimes in which standard convex relaxations fail to achieve the statistical limit \citep{caiMaWuSparsePCA2013,deshpandeMontanariInformationSparsePCA2014,krauthgamerNadlerVilenchikSparsePCA2015,deshpandeMontanariSparsePCA2016}. Computational lower bounds have been obtained under planted-clique assumptions and the Small-Set Expansion hypothesis, through sum-of-squares lower bounds, and via more refined average-case reductions \citep{berthetRigolletComplexity2013,maWigdersonSparsePCA2015,wangBerthetSamworthSparsePCA2016,magdonIsmailNPHardness2017,chanApproximabilitySparsePCA2016,brennanBreslerHuleihel2018,brennanBreslerSparsePCA2019,dingKuniskyWeinBandeiraSparsePCA2024}.

A closely related line of work concerns certification of the restricted isometry property. Hidden-clique reductions, worst-case complexity, average-case hardness, bicriteria certification, and time--accuracy trade-offs have all been investigated \citep{koiranHiddenCliquesCertification2014,natarajanComputationalComplexityCertifying2014,wangAverageCaseRIP2016,weedApproximatelyCertifyingRestricted2018,dingAverageCaseTime2021}. Since the restricted isometry constant can be written as a sparse spectral norm of $X^\top X-I$, these works provide especially close precedents for our sparse spectral-norm result. Section~\ref{subsec_sparse_literature} gives the more quantitative comparison, including a strengthened consequence that we derive from the Koiran--Zouzias construction.

\paragraph{Cut norms, Grothendieck inequalities, and discrete multilinear optimization.}
Matrix cut norms belong to a broader family of discrete bilinear optimization problems governed by Grothendieck-type inequalities. Constant-factor approximation of the ordinary matrix cut norm follows from this theory \citep{alon2006approximating}, while related algorithmic and hardness results for quadratic and bilinear optimization have developed several extensions of Grothendieck's inequality \citep{charikarWirthGrothendieck2004,kindlerNaorSchechtmanGrothendieck2010,khotNaorGrothendieck2012}. Higher-order analogues lead naturally to Boolean multilinear optimization and tensor cut norms; polynomial-time approximation algorithms for such discrete multilinear objectives are studied in \cite{heApproximationAlgorithmsDiscrete2013}, with related tensor-approximation techniques appearing in \cite{bhattiproluWeakDecoupling2017}. The normalized sparse matrix cut norm considered here differs from the ordinary cut norm by its cardinality restriction and normalization, while the tensor cut norm introduces genuinely higher-order mode-wise interactions. Thus the classical constant-factor matrix theory does not directly provide constant-factor guarantees for the two constrained or higher-order functionals studied here.

\subsection{Notation}

For a positive integer $m$, write $[m]=\{1,\ldots,m\}$, and let $\mathbb R_+:=[0,\infty)$.
For a set $S\subseteq[p]$, let $\mathbf{1}_S$ denote its indicator vector and $|S|$ its cardinality, while $I_{\{A\}}$ denotes the indicator function of an event $A$.
For a matrix $M$ and an index set $T\subseteq[p]$, let $M_{T,T}$ denote the principal submatrix of $M$ indexed by $T$.
We write $e_i$ for the $i$th standard basis vector and $\operatorname{supp}(u)$ for the support of a vector $u$.
We use $\|\cdot\|_2$, $\|\cdot\|_0$, and $\|\cdot\|_{\mathrm{op}}$ for the Euclidean norm, support size, and matrix operator norm, respectively, and write $\mathbb S^{p-1}:=\{x\in\mathbb R^p:\|x\|_2=1\}$.
The notation $\langle\cdot,\cdot\rangle$ denotes the standard Euclidean or Frobenius inner product, as appropriate, $\otimes$ denotes the tensor product, and $x^{\otimes d}:=x\otimes\cdots\otimes x$ denotes the $d$-fold tensor product of $x$ with itself.
For a likelihood ratio $L_p$ under a reference distribution $\mathbb Q_p$, $(L_p)_{\le D}$ denotes its orthogonal projection in $L^2(\mathbb Q_p)$ onto the space of polynomials of total degree at most $D$ in the observed data, and $\|\cdot\|_{L^2(\mathbb Q_p)}$ denotes the corresponding $L^2$ norm, namely $\|f\|_{L^2(\mathbb Q_p)}=(\mathbb E_{\mathbb Q_p}[f^2])^{1/2}$. The sub-Gaussian norm of a random variable $X$ is $\|X\|_{\psi_2}:=\inf\{t>0:\E\exp(X^2/t^2)\le2\}$. We write $\polylog(p)$ for $(\log p)^{C}$ with an unspecified constant $C>0$.
For two positive sequences $a_p$ and $b_p$, we write $a_p\asymp b_p$ if $a_p/b_p$ is bounded above and below by positive constants.
A subscript $d$ on asymptotic notation, as in $\asymp_d$, $O_d(\cdot)$, or $\Omega_d(\cdot)$, indicates that the implicit constants may depend on the fixed tensor order $d$.
We use $O(\cdot)$, $o(\cdot)$, and $O_{\mathbb P}(\cdot)$ with their standard asymptotic meanings.
All other notation will be defined at its first appearance.

\subsection{Organization}
Section~\ref{sec:tensor} introduces the tensor spectral-norm problem, reviews the cumulant reverse gap, states the general reduction, and derives the tensor spectral and nuclear consequences. Section~\ref{sec:sparse} develops the planted-clique reverse gaps and lower bounds for the sparse spectral and normalized sparse matrix cut norms. Section~\ref{sec:tensor_cut} introduces the tensor cut norm, proves the spiked-tensor low-degree bound, and derives the corresponding conditional approximation lower bound. The appendices contain the polynomial-threshold extensions, the formal spectral--nuclear duality argument, the quantitative refinement of the Koiran--Zouzias reduction, the sparse certificate constructions, and the remaining proofs.

\section{Motivating Example and the General Framework}\label{sec:tensor}

We begin with the high-order cumulant model of \cite{tang2026detectionharderestimationcertain}. In that model, the low-degree likelihood ratio remains bounded at a signal scale for which an empirical cumulant tensor can already be estimated with substantially smaller error in the tensor spectral norm. This previously established reverse gap provides a direct motivating example for the reduction developed below.

\subsection{Tensor spectral norm and approximation guarantees}

We say that an order-$d$ tensor $\bcT$ is \emph{symmetric} if $\bcT_{i_1,\ldots,i_d}=\bcT_{i_{\pi(1)},\ldots,i_{\pi(d)}}$ for every permutation $\pi$ of $\{1,\ldots,d\}$. We write $\operatorname{Sym}^d(\mathbb R^p)$ for the vector space of symmetric order-$d$ tensors. For $\bcT\in\operatorname{Sym}^d(\mathbb R^p)$, its spectral norm is
\begin{equation*}\label{eq_tensor_spec_norm}
\|\bcT\|_s
:=
\sup_{\|x\|_2=1}
\bigl|\langle \bcT,x^{\otimes d}\rangle\bigr|.
\end{equation*}
This is the natural higher-order analogue of the matrix operator norm and is central in multilinear optimization, tensor decomposition, and polynomial optimization over the sphere; see, for example, \cite{limSingularValuesEigenvalues2005,koldaTensorDecompositionsApplications2009,harrowTestingProductStates2013}. In contrast with the matrix case, computing or approximating tensor spectral norms is already NP-hard for order three and above \citep{hillarMostTensorProblems2013a}.

Let $f_p:\operatorname{Sym}^d(\mathbb R^p)\to\mathbb R_+$ be a map intended to approximate $\|\bcT\|_s$. We consider the following three guarantees.
\begin{itemize}[leftmargin=1.8em]
\item \textbf{Two-sided approximation:} there exist $\rho_{f,p},\zeta_{f,p}\ge1$ such that
\begin{equation}\label{eq_tensor_two_sided}
\rho_{f,p}^{-1}\|\bcT\|_s
\le
f_p(\bcT)
\le
\zeta_{f,p}\|\bcT\|_s,
\qquad
\bcT\in\operatorname{Sym}^d(\mathbb R^p),
\end{equation}
with distortion $\gamma_{f,p}:=\rho_{f,p}\zeta_{f,p}$.
\item \textbf{Lower certificate:}
\begin{equation}\label{eq_tensor_lower}
\gamma_{f,p}^{-1}\|\bcT\|_s
\le
f_p(\bcT)
\le
\|\bcT\|_s,
\qquad
\bcT\in\operatorname{Sym}^d(\mathbb R^p).
\end{equation}
\item \textbf{Upper certificate:}
\begin{equation}\label{eq_tensor_upper}
\|\bcT\|_s
\le
f_p(\bcT)
\le
\gamma_{f,p}\|\bcT\|_s,
\qquad
\bcT\in\operatorname{Sym}^d(\mathbb R^p).
\end{equation}
\end{itemize}
For a general norm $\|\cdot\|$ on a space $\mathcal F_p$, we use the same terminology: a \emph{lower certificate} satisfies $\gamma_{f,p}^{-1}\|T\|\le f_p(T)\le\|T\|$ and an \emph{upper certificate} satisfies $\|T\|\le f_p(T)\le\gamma_{f,p}\|T\|$ for every $T\in\mathcal F_p$.
Lower certificates include algorithms that output a candidate vector $x\in\mathbb S^{p-1}$ and return $|\langle\bcT,x^{\otimes d}\rangle|$, as well as covering-based methods and rounded sum-of-squares procedures. Upper certificates include raw sum-of-squares relaxations, which naturally produce certified upper bounds.

A substantial literature develops polynomial-time approximation algorithms for tensor norms, including matrix unfoldings, covering nets, and sum-of-squares relaxations; see, among others, \cite{brandaoEstimatingOperatorNorms2015,bhattiproluWeakDecoupling2017,bhattiproluSumofSquaresCertificates2017,hopkinsTensorPCA2015,heApproximatingTensorNorms2023}. Existing NP-hardness results rule out exact computation and some forms of approximation \citep{hillarMostTensorProblems2013a,harrowTestingProductStates2013}, but they do not identify the precise polynomial dependence of the best possible worst-case distortion on $p$ and $d$.

\subsection{An established reverse detection--estimation gap}\label{subsec:cumulant_reverse_gap}

Set
\begin{equation}\label{eq_choice_n_a}
n=\lfloor p^{d-1}/2\rfloor,
\qquad
a=\frac{p^{1/2}}{n^{1/d}(\log p)^9}.
\end{equation}
For every fixed $d\ge3$, this choice satisfies $p^{d/2}<n<p^{d-1}$ and $a=o(1)$ for all sufficiently large $p$. Under the null law $\mathbb Q_p$, the observations are $Y_i=Z_i$ with $Z_i\stackrel{\mathrm{i.i.d.}}{\sim}\mathcal N(0,I_p)$. Under the alternative law $\mathbb P_p$, first draw $U \in \R^{p}$ with i.i.d.\ Rademacher coordinates, so that $\|U\|_2=\sqrt p$, and conditionally on $U$ let
\[
X_i=\sqrt{a/p}\,W_iU+Z_i,
\qquad
Y_i=S_UX_i,
\qquad
i\in[n],
\]
where the $W_i$ are i.i.d.\ draws from the distribution constructed in Lemma \ref{lemma_existence_1d} and have $O_d(1)$ sub-Gaussian norm. The $Z_i$ are i.i.d.\ draws from $\mathcal N(0,I_p)$, and $U$, $(W_i)_{i\in[n]}$, and $(Z_i)_{i\in[n]}$ are mutually independent. Moreover,
\begin{equation*}\label{eq_whitening}
S_U
=
I_p-\frac{a}{1+a+\sqrt{1+a}}\frac{UU^\top}{p}.
\end{equation*}
The matrix $S_U$ whitens the covariance of $X_i$. Let $L_p=d\mathbb P_p/d\mathbb Q_p$.

\begin{conjecture}[Model-specific low-degree conjecture]\label{conj_low_degree}
Fix $d\ge3$ and consider the cumulant testing family above. 
There exists a constant $c_d>0$, depending only on $d$, such that if
\begin{equation}\label{eq_tensor_ldlr_conjecture_antecedent}
\bigl\|(L_p)_{\le\lfloor c_d(\log p)^2\rfloor}\bigr\|_{L^2(\mathbb Q_p)}
=
O(1),
\end{equation}
where $\|f\|_{L^2(\mathbb Q_p)}=(\mathbb E_{\mathbb Q_p}[f^2])^{1/2}$, 
then no sequence of tests generated by a single uniform polynomial-time algorithm (possibly randomized) strongly distinguishes $\mathbb P_p$ from $\mathbb Q_p$. Equivalently, every such test sequence satisfies
\begin{equation}\label{eq_tensor_polytime_strong_hardness}
\limsup_{p\to\infty}
\left[
\mathbb Q_p(\phi_p=1)+\mathbb P_p(\phi_p=0)
\right]
>
0.
\end{equation}
\end{conjecture}

For suitably structured testing families, boundedness of the low-degree likelihood-ratio norm through polylogarithmic degree is used as evidence against polynomial-time strong detection; see \cite{hopkins2018statistical,kunisky2019notes} for foundational formulations and \cite{schramm2022computational} for related uses of low-degree polynomials as a proxy.

Let $\widehat{\bcM}_j=n^{-1}\sum_{i=1}^nY_i^{\otimes j}$ and let $\widehat{\bcK}_d$ be the plug-in sample cumulant formed from the empirical moment tensors.

\begin{proposition}[Cumulant reverse gap]\label{prop_cumulant_reverse_gap}
For every fixed $d\ge3$ and every fixed $c>0$, the testing family above satisfies
\[
\bigl\|(L_p)_{\le\lfloor c(\log p)^2\rfloor}\bigr\|_{L^2(\mathbb Q_p)}
=
O(1).
\]
For a sufficiently large constant $C_d>0$ depending only on $d$, define
\begin{equation}\label{eq_cumulant_signal_radius}
s_p
=
|\kappa_d(W_1)|
\left(\frac{a}{1+a}\right)^{d/2}
\asymp_d
\frac{p^{d/4}}{n^{1/2}(\log p)^{9d/2}},
\qquad
r_p
=
C_d\frac{p^{d/2}}{n},
\end{equation}
where $\kappa_d(W_1)$ denotes the $d$th cumulant of $W_1$.
Then
\[
\mathbb Q_p\bigl(\|\widehat{\bcK}_d\|_s\le r_p\bigr)\longrightarrow1,
\qquad
\mathbb P_p\bigl(\|\widehat{\bcK}_d-\bcK_d(Y_i\mid U)\|_s\le r_p\bigr)\longrightarrow1,
\]
and $\|\bcK_d(Y_i\mid U)\|_s=s_p$ for every realization of $U$. Consequently,
\begin{equation*}\label{eq_cumulant_reverse_gap}
\frac{s_p}{r_p}
\asymp_d
\frac{p^{d/4-1/2}}{(\log p)^{9d/2}}.
\end{equation*}
\end{proposition}

The proof is collected in Section~\ref{sec_proof}. The essential point is that the empirical cumulant is already accurate at radius $r_p$, while Conjecture~\ref{conj_low_degree} places the signal scale $s_p$ below the conjectured polynomial-time detection threshold.

\subsection{The general reverse-gap reduction}\label{sec:framework}
For each $p$, let $\mathbb Q_p$ and $\mathbb P_p$ denote the null and alternative laws of the observed data. The alternative may be a mixture over a latent parameter. Let $\mu_p$ take values in a normed vector space $\mathcal F_p$, and assume
\begin{equation}\label{eq_general_separation}
\mu_p=0
\quad\text{under }\mathbb Q_p,
\qquad
\|\mu_p\|\ge s_p
\quad\text{under }\mathbb P_p,
\end{equation}
where $s_p>0$ and the second statement holds almost surely with respect to any latent randomness under $\mathbb P_p$.

Let $\mathcal A$ be a prescribed class of sequences $(\phi_p)_{p\ge1}$ of possibly randomized $\{0,1\}$-valued tests. In computational applications, membership in $\mathcal A$ requires the entire sequence to be generated by a single uniform algorithm. We say that $(\phi_p)$ \emph{strongly distinguishes} $\mathbb P_p$ from $\mathbb Q_p$ if
\begin{equation*}\label{eq_strong_distinguishability}
\mathbb Q_p(\phi_p=1)+\mathbb P_p(\phi_p=0)
\longrightarrow0.
\end{equation*}
We impose the strong-detection hardness condition that no sequence in $\mathcal A$ strongly distinguishes. Equivalently,
\begin{equation}\label{eq_no_strong_distinguishability}
\limsup_{p\to\infty}
\left[
\mathbb Q_p(\phi_p=1)+\mathbb P_p(\phi_p=0)
\right]
>
0
\qquad
\text{for every }(\phi_p)_{p\ge1}\in\mathcal A.
\end{equation}

Let $\widehat\mu_p$ be an estimator and let $r_p>0$ be a deterministic estimation radius such that
\begin{equation}\label{eq_estimation_radius}
\mathbb Q_p\bigl(\|\widehat\mu_p\|\le r_p\bigr)\longrightarrow1,
\qquad
\mathbb P_p\bigl(\|\widehat\mu_p-\mu_p\|\le r_p\bigr)\longrightarrow1.
\end{equation}
Let $\mathcal B$ be a nonempty class of sequences $(f_p)_{p\ge1}$ of norm approximators $f_p:\mathcal F_p\to\mathbb R_+$. In computational applications, every sequence in $\mathcal B$ is required to be generated by a single uniform algorithm. Each $(f_p)_{p\ge1}\in\mathcal B$ is equipped with deterministic approximation factors $\rho_{f,p},\zeta_{f,p}\ge1$ satisfying
\begin{equation*}\label{eq_general_approx}
\rho_{f,p}^{-1}\|T\|
\le
f_p(T)
\le
\zeta_{f,p}\|T\|,
\qquad
T\in\mathcal F_p,
\end{equation*}
and we write $\gamma_{f,p}:=\rho_{f,p}\zeta_{f,p}$.

We require the following threshold-closure property: for every $(f_p)_{p\ge1}\in\mathcal B$, the test sequence
\begin{equation}\label{eq_closure}
\phi_p
=
I_{\{f_p(\widehat\mu_p)>s_p/(2\rho_{f,p})\}}
\end{equation}
belongs to $\mathcal A$. In computational applications, we say that the approximation factors of a family $(f_p)_{p\ge1}\in\mathcal B$ are \emph{efficiently computable} if the uniform algorithm generating the family also outputs $\rho_{f,p}$ and $\zeta_{f,p}$ on input $p$, within the resource bounds defining $\mathcal B$. 
The threshold-closure property then requires a single algorithm, operating within the resource bounds defining $\mathcal A$, to compute the approximation factors and the displayed separating thresholds for every $p$.

\begin{theorem}[General reduction to approximation lower bounds]\label{thm_framework}
Suppose that \eqref{eq_general_separation}, \eqref{eq_no_strong_distinguishability}, \eqref{eq_estimation_radius}, and the threshold-closure property \eqref{eq_closure} hold. Then every $(f_p)_{p\ge1}\in\mathcal B$ satisfies
\begin{equation}\label{eq_framework_lb}
\limsup_{p\to\infty}
\frac{\gamma_{f,p}r_p}{s_p}
\ge
\frac14.
\end{equation}
The same conclusion applies to lower certificates by setting $\zeta_{f,p}=1$ and to upper certificates by setting $\rho_{f,p}=1$.
\end{theorem}

\begin{proof}
Fix $(f_p)_{p\ge1}\in\mathcal B$ and suppose, toward a contradiction, that \eqref{eq_framework_lb} fails. Then
\begin{equation}\label{eq_framework_forbidden}
\gamma_{f,p}
<
\frac{s_p}{4r_p}
\end{equation}
for all sufficiently large $p$. By threshold closure,
\[
\phi_p
=
I_{\{f_p(\widehat\mu_p)>s_p/(2\rho_{f,p})\}}
\]
belongs to $\mathcal A$.

On the event $\{\|\widehat\mu_p\|\le r_p\}$ under $\mathbb Q_p$,
\[
f_p(\widehat\mu_p)
\le
\zeta_{f,p}r_p
=
\frac{\gamma_{f,p}r_p}{\rho_{f,p}}
<
\frac{s_p}{4\rho_{f,p}}
<
\frac{s_p}{2\rho_{f,p}}.
\]
Hence $\mathbb Q_p(\phi_p=1)\to0$.

Because $\gamma_{f,p}\ge1$, \eqref{eq_framework_forbidden} implies $r_p<s_p/4$ for all sufficiently large $p$. On the event $\{\|\widehat\mu_p-\mu_p\|\le r_p\}$ under $\mathbb P_p$,
\[
f_p(\widehat\mu_p)
\ge
\rho_{f,p}^{-1}\|\widehat\mu_p\|
\ge
\rho_{f,p}^{-1}\bigl(\|\mu_p\|-r_p\bigr)
>
\frac{3s_p}{4\rho_{f,p}}
>
\frac{s_p}{2\rho_{f,p}}.
\]
Hence $\mathbb P_p(\phi_p=0)\to0$, contradicting \eqref{eq_no_strong_distinguishability}.
\end{proof}

Theorem~\ref{thm_framework} is informative whenever $s_p/r_p\to\infty$. It then rules out $\gamma_{f,p}=o(s_p/r_p)$ for every admissible approximator family.

\subsection{Computational lower bound for tensor spectral norm}

Applying Theorem~\ref{thm_framework} to the cumulant testing family, with the scales provided by Proposition~\ref{prop_cumulant_reverse_gap}, yields the following lower bound. The proof is collected in Section~\ref{sec_proof}.
\begin{theorem}[Conditional lower bound for tensor spectral-norm approximation]\label{thm_main}
Fix $d\ge3$ and assume Conjecture~\ref{conj_low_degree}. Let $(f_p)_{p\ge1}$ be a family of tensor spectral norm approximators generated by a single uniform deterministic polynomial-time algorithm and satisfying one of \eqref{eq_tensor_two_sided}, \eqref{eq_tensor_lower}, or \eqref{eq_tensor_upper}. Assume that its approximation factors are efficiently computable and that the ratios $s_p/(2\rho_{f,p})$ are uniformly polynomial-time computable. 
Then we have
\[
\limsup_{p\to\infty}
\gamma_{f,p}
\frac{(\log p)^{9d/2}}{p^{d/4-1/2}}
>
0.
\]
\end{theorem}

\begin{remark}[Polynomial-threshold extension]\label{rem_polynomial_threshold_extension}
Question~4.2 in Section~4.1 of \cite{kunisky2019notes} asks whether boundedness of the low-degree likelihood-ratio norm through degree $D$ rules out strong distinction by thresholding a degree-$D$ polynomial. An affirmative answer for the present cumulant testing family supplies the strong-detection hardness condition for the polynomial-threshold test class. Under that additional assumption, the lower bound holds without Conjecture~\ref{conj_low_degree} or a running-time restriction for every approximator of polynomial-threshold degree at most $c_d(\log p)^2$. Section~\ref{sec:polynomial_threshold_approximators} gives the precise definition, statement, and proof.
\end{remark}

\begin{remark}[Tightness and known upper bounds]
    Several independent results suggest that the exponent $d/4-1/2$ in Theorem~\ref{thm_main} captures the correct computational scale. For entrywise nonnegative tensors, polynomial-time algorithms achieve distortion of order $p^{d/4-1/2}$ \citep{bhattiproluWeakDecoupling2017}. For random Rademacher or Gaussian tensors, SoS-based certification results and matching or nearly matching integrality gaps exhibit the same scale \citep{bhattiproluSumofSquaresCertificates2017,hopkins2017power}. Although these results concern restricted input classes rather than globally valid worst-case approximators, they provide independent evidence that the exponent in our conditional lower bound is tight.

    For arbitrary signed tensors in the worst case, however, the best known polynomial-time guarantees remain of order $p^{d/2-1}=p^{(d-2)/2}$ up to logarithmic factors, obtainable through matrix unfolding or polynomial-size net constructions \citep{bhattiproluWeakDecoupling2017,heApproximatingTensorNorms2023}. A closely related separation already appears in the SoS literature: random tensors exhibit tight certification gaps at the $p^{d/4-1/2}$ scale, whereas general worst-case rounding guarantees for arbitrary tensors remain at the larger $p^{d/2-1}$ scale.
\end{remark}

\subsection{A duality consequence for tensor nuclear norm}\label{subsec:tensor_nuclear_duality}

For $\bcT\in\operatorname{Sym}^d(\mathbb R^p)$, define the projective tensor nuclear norm by
\begin{equation*}\label{eq_tensor_nuclear_norm}
\|\bcT\|_*
:=
\inf\left\{
\sum_{\ell=1}^r |a_\ell|:
\bcT=\sum_{\ell=1}^r a_\ell x_\ell^{\otimes d},\quad
r\in\mathbb N,\quad
\|x_\ell\|_2=1\ \text{for every }\ell\in[r]
\right\}.
\end{equation*}
For symmetric tensors, restricting to symmetric rank-one decompositions entails no loss \citep{friedlandNuclearNormHigherOrder2018}. Its dual norm on $\operatorname{Sym}^d(\mathbb R^p)$ is the tensor spectral norm:
\begin{equation}\label{eq_tensor_spectral_nuclear_duality}
\|\bcT\|_*
=
\sup_{\|\bcS\|_s\le1}|\langle\bcT,\bcS\rangle|,
\qquad
\|\bcS\|_s
=
\sup_{\|\bcT\|_*\le1}|\langle\bcS,\bcT\rangle|.
\end{equation}

For a norm $g_p$ on $\operatorname{Sym}^d(\mathbb R^p)$, define its dual norm by
$$
g_p^\circ(\bcS)
:=
\sup_{g_p(\bcT)\le1}|\langle\bcS,\bcT\rangle|.
$$
Suppose that, for some $\rho_{g,p},\zeta_{g,p}\ge1$,
\begin{equation}\label{eq_tensor_nuclear_approximation}
\rho_{g,p}^{-1}\|\bcT\|_*
\le
g_p(\bcT)
\le
\zeta_{g,p}\|\bcT\|_*,
\qquad
\bcT\in\operatorname{Sym}^d(\mathbb R^p).
\end{equation}
The corresponding unit balls satisfy
\[
\zeta_{g,p}^{-1}\{\bcT:\|\bcT\|_*\le1\}
\subseteq
\{\bcT:g_p(\bcT)\le1\}
\subseteq
\rho_{g,p}\{\bcT:\|\bcT\|_*\le1\}.
\]
Taking support functions and using \eqref{eq_tensor_spectral_nuclear_duality} gives
\begin{equation}\label{eq_dual_spectral_approximation}
\zeta_{g,p}^{-1}\|\bcS\|_s
\le
g_p^\circ(\bcS)
\le
\rho_{g,p}\|\bcS\|_s.
\end{equation}
Thus dualization preserves the distortion while interchanging the two approximation factors.

We call a family $(g_p)_{p\ge1}$ \emph{efficiently dualizable} if both $(g_p)_{p\ge1}$ and $(g_p^\circ)_{p\ge1}$ are uniformly computable by deterministic polynomial-time algorithms. This is precisely the computational closure property needed to transfer Theorem~\ref{thm_main} through duality. In the weak-oracle model, computation of a norm and computation of its dual are polynomial-time interreducible \citep{friedlandComputationalComplexityDuality2016}.

Let
\begin{equation*}\label{eq_cumulant_signal_nuclear_corollary}
s_p^{\mathrm{cum}}
:=
|\kappa_d(W_1)|
\left(\frac{a}{1+a}\right)^{d/2}
\end{equation*}
denote the cumulant signal level from \eqref{eq_cumulant_signal_radius}.

\begin{corollary}[Conditional tensor nuclear-norm lower bound]\label{cor_tensor_nuclear_lower_bound}
Fix $d\ge3$ and assume Conjecture~\ref{conj_low_degree}. Let $(g_p)_{p\ge1}$ be an efficiently dualizable family of norms on $\operatorname{Sym}^d(\mathbb R^p)$ satisfying \eqref{eq_tensor_nuclear_approximation}, and let
$
\gamma_{g,p}:=\rho_{g,p}\zeta_{g,p}.
$
Assume that the approximation factors are efficiently computable and that the ratios $s_p^{\mathrm{cum}}/(2\zeta_{g,p})$ are uniformly polynomial-time computable. Then
\begin{equation}\label{eq_tensor_nuclear_lower_bound}
\limsup_{p\to\infty}
\gamma_{g,p}
\frac{(\log p)^{9d/2}}{p^{d/4-1/2}}
>
0.
\end{equation}
In particular, no such family can achieve distortion
\[
o\left(
\frac{p^{d/4-1/2}}{(\log p)^{9d/2}}
\right).
\]
\end{corollary}

By \eqref{eq_dual_spectral_approximation}, the dual family $(g_p^\circ)_{p\ge1}$ approximates the tensor spectral norm with lower and upper approximation factors $\zeta_{g,p}$ and $\rho_{g,p}$, respectively, and hence with the same distortion $\gamma_{g,p}$. Moreover, the separating threshold required by Theorem~\ref{thm_main} becomes exactly $s_p^{\mathrm{cum}}/(2\zeta_{g,p})$. The corollary therefore follows directly from Theorem~\ref{thm_main}. A formal proof is given in Section~\ref{sec:tensor_nuclear_duality}.

For comparison, \cite{friedlandNuclearNormHigherOrder2018} established NP-hardness results for computing the projective tensor nuclear norm, including approximation to prescribed accuracy, but these results do not provide a dimension-dependent multiplicative lower bound. On the algorithmic side, tensor-partition methods yield polynomial-time distortion $O_d(p^{(d-2)/2})$ in equal dimensions \citep{liBoundsSpectralNuclear2016}. The deterministic sphere-covering algorithms of \cite{heApproximatingTensorNorms2023} improve this general worst-case guarantee for both the tensor spectral and nuclear norms to
\[
O_d\left(
\left(\frac{p}{\log p}\right)^{(d-2)/2}
\right).
\]
Thus a polynomial gap remains between the conditional lower scale in Corollary~\ref{cor_tensor_nuclear_lower_bound} and the best known general worst-case upper scale.

\section{Sparse Spectral Norm and Normalized Sparse Matrix Cut Norm}\label{sec:sparse}

We next use the planted-clique model to produce two reverse detection--estimation gaps from the same centered adjacency estimator.

\subsection{The reverse detection--estimation gaps}\label{subsec:sparse_reverse_gaps}

For a real symmetric matrix $M\in\mathbb R^{p\times p}$ and an integer $2\le k\le p$, define the {\it $k$-sparse spectral norm} by
\begin{equation}\label{eq_sparse_spectral_norm}
\|M\|_{\mathrm{sp},k}
:=
\sup_{\substack{\|u\|_2=1\\ \|u\|_0\le k}}
|u^\top Mu|
=
\max_{\substack{T\subseteq[p]\\1\le|T|\le k}}
\|M_{T,T}\|_{\mathrm{op}}.
\end{equation}

For an arbitrary real matrix $M\in\mathbb R^{p\times p}$, define the {\it normalized $k$-sparse matrix cut norm} by
\begin{equation}\label{eq_sparse_cut_norm}
\|M\|_{\square,k}
:=
\max_{\substack{R,C\subseteq[p]\\1\le|R|,|C|\le k}}
\frac{\bigl|\mathbf 1_R^\top M\mathbf 1_C\bigr|}{\sqrt{|R||C|}}.
\end{equation}
Equivalently, each indicator vector is divided by its Euclidean norm. This normalization puts the planted signal and stochastic error on the same scale as in the sparse spectral norm. We first record that both functionals are indeed norms; the elementary proofs are in Sections~\ref{sec_proof_lem_sparse_spectral_norm} and~\ref{sec_proof_lem_sparse_cut_norm}.

\begin{lemma}\label{lem_sparse_spectral_norm}
For every $k\ge2$, $\|\cdot\|_{\mathrm{sp},k}$ is a norm on the vector space of real symmetric $p\times p$ matrices.
\end{lemma}

\begin{lemma}\label{lem_sparse_cut_norm}
For every $k\ge1$, $\|\cdot\|_{\square,k}$ is a norm on $\mathbb R^{p\times p}$.
\end{lemma}

Under the null distribution $\mathbb Q_p^{\mathrm{PC}}$, let $A$ be the adjacency matrix of a graph sampled from the Erd\H{o}s--R\'enyi model $G(p,1/2)$, in which each edge is independently present with probability $1/2$.
Under the alternative distribution $\mathbb P_p^{\mathrm{PC}}$, first draw $S\subseteq[p]$ uniformly among all subsets of size $k$, set $A_{ij}=1$ for all distinct $i,j\in S$, and draw every remaining edge independently with probability $1/2$. The diagonal entries are zero under both distributions.

\begin{assumption}[Planted-clique hardness]\label{assump_planted_clique}
For the sequence $k=k_p$ under consideration, the clique size is known to the testing algorithm. No sequence of tests generated by a single uniform randomized polynomial-time algorithm, given the observed graph and $k_p$, strongly distinguishes $\mathbb P_p^{\mathrm{PC}}$ from $\mathbb Q_p^{\mathrm{PC}}$.
\end{assumption}

The polynomial-regime version takes $k=\lceil p^{1/2-\varepsilon}\rceil$ for an arbitrary fixed $\varepsilon\in(0,1/2)$. This hardness assumption and its implications for high-dimensional statistical problems have been studied extensively; see, among others, \cite{berthetRigolletComplexity2013}.

Consider the centered adjacency matrix
\begin{equation*}\label{eq_centered_adjacency}
\widehat\mu
=
A-\frac12(J-I),
\end{equation*}
where $J$ is the all-ones matrix and $I$ is the identity. Under $\mathbb Q_p^{\mathrm{PC}}$, the parameter is $\mu=0$. Under $\mathbb P_p^{\mathrm{PC}}$, conditional on $S$,
\begin{equation*}\label{eq_planted_clique_parameter}
\mu_S
=
\E(\widehat\mu\mid S)
=
\frac12\left(\mathbf 1_S\mathbf 1_S^\top-I_S\right),
\end{equation*}
where $I_S$ is diagonal with diagonal $\mathbf 1_S$.

The nonzero principal block of $\mu_S$ is $\frac12(J_k-I_k)$, whose eigenvalues are $(k-1)/2$ and $-1/2$. Hence
\begin{equation*}\label{eq_planted_clique_signal}
\|\mu_S\|_{\mathrm{sp},k}
=
\frac{k-1}{2}.
\end{equation*}
The normalized sparse cut norm has the same signal size. Indeed, taking $R=C=S$ gives the lower bound $(k-1)/2$, while normalized indicators are unit vectors and therefore
\[
\frac{|\mathbf 1_R^\top\mu_S\mathbf 1_C|}{\sqrt{|R||C|}}
\le
\|\mu_S\|_{\mathrm{op}}
=
\frac{k-1}{2}.
\]
Thus
\begin{equation}\label{eq_planted_clique_cut_signal}
\|\mu_S\|_{\square,k}
=
\frac{k-1}{2}.
\end{equation}
The estimation radius for both norms is controlled by the following concentration bounds; the proof, by standard net and union-bound arguments, is given in Section~\ref{sec_proof_thm_sparse_norm_concentration}.

\begin{theorem}[Concentration of the sparse spectral and cut norms]\label{thm_sparse_norm_concentration}
There exist universal constants $C,c>0$ such that, under $\mathbb Q_p^{\mathrm{PC}}$ and under $\mathbb P_p^{\mathrm{PC}}$ conditional on any planted set $S$,
\begin{equation*}
\mathbb P\left(
\max\left\{
\|\widehat\mu-\mu\|_{\mathrm{sp},k},
\|\widehat\mu-\mu\|_{\square,k}
\right\}
>
C\sqrt{k\log\left(\frac{ep}{k}\right)}
\right)
\le
2\exp\left(
-ck\log\left(\frac{ep}{k}\right)
\right).
\end{equation*}
Here $\mu=0$ under the null and $\mu=\mu_S$ under the alternative.
\end{theorem}

With $C$ chosen as in Theorem \ref{thm_sparse_norm_concentration}, define
\begin{equation}\label{eq_planted_clique_estimation_radius}
r_p
=
C\sqrt{k\log\left(\frac{ep}{k}\right)},
\qquad
s_p
=
\frac{k-1}{2}.
\end{equation}
For both norms,
\begin{equation}\label{eq_planted_clique_reverse_gap}
\frac{s_p}{r_p}
\asymp
\sqrt{\frac{k}{\log(ep/k)}}.
\end{equation}

\subsection{Computational lower bounds}\label{subsec:sparse_lower_bounds}

Combining the reverse gap \eqref{eq_planted_clique_reverse_gap} with the reduction of Section~\ref{sec:framework} yields conditional lower bounds for both norms; the proofs are in Sections~\ref{sec_proof_thm_sparse_spectral_planted_clique} and~\ref{sec_proof_thm_sparse_cut_planted_clique}.
\begin{theorem}[Conditional sparse spectral norm lower bound]\label{thm_sparse_spectral_planted_clique}
Suppose Assumption~\ref{assump_planted_clique} holds and $k/\log p\to\infty$. Let $(f_{p,k})_{p\ge1}$ be generated by a single uniform randomized polynomial-time algorithm such that, for every real symmetric $p\times p$ matrix $M$,
\begin{equation}\label{eq_sparse_spectral_approximation}
\rho_{f,p}^{-1}\|M\|_{\mathrm{sp},k}
\le
f_{p,k}(M)
\le
\zeta_{f,p}\|M\|_{\mathrm{sp},k}
\end{equation}
with probability at least $2/3$, where $\rho_{f,p},\zeta_{f,p}\ge1$ and $\gamma_{f,p}=\rho_{f,p}\zeta_{f,p}$. Define $\tau_{p,k}=s_p/(2\rho_{f,p})$. Assume the factors are efficiently computable and $\tau_{p,k}$ is uniformly polynomial-time computable from $(p,k)$. Then there is a universal constant $c>0$ such that
\begin{equation}\label{eq_sparse_spectral_lower_bound}
\limsup_{p\to\infty}
\gamma_{f,p}
\sqrt{\frac{\log(ep/k)}{k}}
\ge
c.
\end{equation}
\end{theorem}

The same argument applies to the normalized sparse cut norm.

\begin{theorem}[Conditional normalized sparse cut-norm lower bound]\label{thm_sparse_cut_planted_clique}
Suppose Assumption~\ref{assump_planted_clique} holds and $k/\log p\to\infty$. Let $(g_{p,k})_{p\ge1}$ be generated by a single uniform randomized polynomial-time algorithm such that, for every real $p\times p$ matrix $M$,
\begin{equation*}\label{eq_sparse_cut_approximation}
\rho_{g,p}^{-1}\|M\|_{\square,k}
\le
g_{p,k}(M)
\le
\zeta_{g,p}\|M\|_{\square,k}
\end{equation*}
with probability at least $2/3$, where $\rho_{g,p},\zeta_{g,p}\ge1$ and $\gamma_{g,p}=\rho_{g,p}\zeta_{g,p}$. Define $\tau_{p,k}^{\square}=s_p/(2\rho_{g,p})$. Assume the factors are efficiently computable and $\tau_{p,k}^{\square}$ is uniformly polynomial-time computable from $(p,k)$. Then there is a universal constant $c>0$ such that
\begin{equation}\label{eq_sparse_cut_lower_bound}
\limsup_{p\to\infty}
\gamma_{g,p}
\sqrt{\frac{\log(ep/k)}{k}}
\ge
c.
\end{equation}
\end{theorem}

For $k=\lceil p^{1/2-\varepsilon}\rceil$ with fixed $\varepsilon\in(0,1/2)$, the two theorems rule out distortion
\[
o\left(
\frac{p^{1/4-\varepsilon/2}}{\sqrt{\log p}}
\right)
\]
for the corresponding uniform randomized polynomial-time approximator families.

\begin{remark}[Near-optimality of the lower bounds]\label{remark:sparse_upper_gap}
    There are deterministic polynomial-time lower certificates satisfying
    \begin{equation*}\label{eq_sparse_spectral_upper_summary}
    \frac1{\sqrt{k}}\|M\|_{\mathrm{sp},k}
    \le
    f_{p,k}^{\mathrm{sp}}(M)
    \le
    \|M\|_{\mathrm{sp},k}
    \end{equation*}
    for every symmetric $M$, and
    \begin{equation*}\label{eq_sparse_cut_upper_summary}
    \frac1{\sqrt{k}}\|M\|_{\square,k}
    \le
    f_{p,k}^{\square}(M)
    \le
    \|M\|_{\square,k}
    \end{equation*}
    for every real $M$. Propositions~\ref{prop_sparse_spectral_upper_bound} and~\ref{prop_sparse_cut_upper_bound} give the constructions and proofs.
    
    Consequently, in every planted-clique-hard polynomial regime $k=p^\alpha$ with $0<\alpha<1/2$, the conditional lower scale for each norm is $\sqrt{k/\log(ep/k)}$, while a deterministic algorithm achieves distortion at most $\sqrt{k}$. The lower and upper asymptotic scales differ only by a factor of order $\sqrt{\log p}$.
\end{remark}

\subsection{Connection with sparse PCA, restricted isometry, cut norms, and existing hardness results}\label{subsec_sparse_literature}

The optimization problem in \eqref{eq_sparse_spectral_norm} is closely related to sparse principal component analysis. For a positive semidefinite matrix $M\succeq0$, the canonical objective
\[
\operatorname{SPCA}_k(M)
:=
\max_{\|u\|_2=1,\,\|u\|_0\le k}
u^\top Mu
\]
equals $\|M\|_{\mathrm{sp},k}$.

\begin{table}[ht]
\centering
\small
\renewcommand{\arraystretch}{1.25}
\resizebox{\textwidth}{!}{
\begin{tabular}{p{3.2cm}p{5.0cm}p{10.0cm}}
\hline
Result
&
Objective and input class
&
Approximation lower bound
\\
\hline
\cite{magdonIsmailNPHardness2017}
&
$\operatorname{SPCA}_k(M)$ for symmetric inputs
&
Rules out every fixed-factor approximation under hidden-clique hardness, but gives no explicit diverging rate in $k$.
\\
\hline
\cite{chanApproximabilitySparsePCA2016}
&
$\operatorname{SPCA}_k(M)$ for positive semidefinite inputs
&
Rules out every fixed-factor approximation under the Small-Set Expansion hypothesis, but gives no explicit diverging rate in $k$.
\\
\hline
\cite{koiranHiddenCliquesCertification2014}
&
$\delta_k(X)=\|X^\top X-I\|_{\mathrm{sp},k}$
&
Under their hidden-clique hypothesis, the stated theorems establish hardness for distinguishing certain polynomially separated RIP parameters. By revisiting and optimizing Propositions~4 and~5, we show that their construction rules out distortion $o(\sqrt{k/\log(ep/k)})$.
\\
\hline
Theorem~\ref{thm_sparse_spectral_planted_clique}
&
$\|M\|_{\mathrm{sp},k}$ for arbitrary symmetric inputs
&
Under randomized planted-clique detection hardness, rules out distortion $o(\sqrt{k/\log(ep/k)})$ for randomized global approximators satisfying the uniformity, factor-computability, and threshold-closure conditions.
\\
\hline
\end{tabular}
}
\caption{Comparison with existing hardness results for sparse PCA, restricted isometry certification, and sparse spectral norm approximation.}
\label{tab:sparse_pca_hardness_comparison}
\end{table}

For sparse PCA, \cite{magdonIsmailNPHardness2017} proved NP-hardness of exact computation and fixed-factor hardness under a hidden-clique assumption, while \cite{chanApproximabilitySparsePCA2016} proved fixed-factor hardness for positive semidefinite inputs under the Small-Set Expansion hypothesis. Neither result gives an explicit diverging approximation factor as a function of $k$.

A more direct connection comes from the restricted isometry constant
\[
\delta_k(X)
:=
\max_{\substack{\|u\|_2=1\\\|u\|_0\le k}}
\left|\|Xu\|_2^2-1\right|
=
\|X^\top X-I\|_{\mathrm{sp},k}.
\]
Thus, approximating the restricted isometry constant is equivalent to approximating the sparse spectral norm on the restricted class $M=X^\top X-I$. The closest predecessor to Theorem~\ref{thm_sparse_spectral_planted_clique} is therefore the Cholesky reduction of \cite{koiranHiddenCliquesCertification2014}.
Their stated theorems establish hardness for distinguishing certain polynomially separated RIP parameters. By optimizing the parameter choices in their Propositions~4 and~5, we show
\[
\limsup_{p\to\infty}
\gamma_{f,p}
\sqrt{\frac{\log(ep/k)}{k}}
>
0.
\]
The complete derivation is in Section~\ref{sec_kz_quantitative_refinement}. Subsequent RIP hardness results strengthen the theory in complementary bicriteria, worst-case, and average-case directions \citep{natarajanComputationalComplexityCertifying2014,weedApproximatelyCertifyingRestricted2018,wangAverageCaseRIP2016,dingAverageCaseTime2021}, but do not directly provide a stronger non-bicriteria multiplicative lower bound at the same sparsity order.

The normalized sparse cut norm \eqref{eq_sparse_cut_norm} is a localized bilinear analogue of the sparse spectral norm. Unlike the ordinary matrix cut norm, it restricts both sides to at most $k$ coordinates and normalizes by the geometric mean of the two support sizes. The unrestricted matrix cut norm admits a constant-factor polynomial-time approximation \citep{alon2006approximating}, but that result does not furnish a constant-factor approximation for the cardinality-constrained normalized functional. Theorem~\ref{thm_sparse_cut_planted_clique} shows that, under planted-clique hardness, its worst-case distortion must diverge with $k$.

\section{Tensor Cut Norm}\label{sec:tensor_cut}

We now turn to a different tensor geometry. Instead of optimizing over Euclidean unit vectors, the tensor cut norm optimizes over independent mode-wise subsets.

\subsection{The reverse detection--estimation gap}\label{subsec:tensor_cut_reverse_gap}

For fixed $d\ge2$ and $T\in\mathbb R^{p\times\cdots\times p}$, define the normalized order-$d$ tensor cut norm by
\begin{equation}\label{eq_tensor_cut_norm}
\|T\|_{\square,d}
:=
p^{-d/2}
\max_{I_1,\ldots,I_d\subseteq[p]}
\left|
\sum_{i_1\in I_1,\ldots,i_d\in I_d}
T_{i_1,\ldots,i_d}
\right|.
\end{equation}
The normalization does not affect multiplicative distortion.

\begin{lemma}\label{lem_tensor_cut_norm}
The functional \eqref{eq_tensor_cut_norm} is a norm on $\mathbb R^{p\times\cdots\times p}$.
\end{lemma}

The following calculation is a direct asymmetric specialization of the standard low-degree analysis for the additive Gaussian spiked-tensor model; see \cite{kunisky2019notes}. We include it to record the precise normalization and degree range needed below.

Consider the asymmetric spiked-tensor model
\begin{equation}\label{eq_tensor_cut_model}
Y=Z
\quad\text{under }\mathbb Q_p^{\mathrm{TC}},
\qquad
Y=\lambda_p\,u^{(1)}\otimes\cdots\otimes u^{(d)}+Z
\quad\text{under }\mathbb P_p^{\mathrm{TC}},
\end{equation}
where $Z$ has i.i.d.\ standard Gaussian entries and the $u^{(j)}$ are independent and uniform on $\{\pm p^{-1/2}\}^p$. We take $\widehat\mu=Y$, with parameter $\mu=0$ under the null and
\[
\mu
=
\lambda_p\,u^{(1)}\otimes\cdots\otimes u^{(d)}
\]
under the alternative.

\begin{lemma}[Signal separation]\label{lem_tensor_cut_signal}
For every realization of the spike,
\begin{equation*}\label{eq_tensor_cut_signal}
\|\mu\|_{\square,d}
\ge
2^{-d}\lambda_p.
\end{equation*}
\end{lemma}

\begin{lemma}[Noise radius]\label{lem_tensor_cut_noise}
For every fixed $d\ge2$, there is a constant $C_d>0$ such that
\begin{equation*}\label{eq_tensor_cut_noise}
\mathbb Q_p^{\mathrm{TC}}\left(
\|Z\|_{\square,d}>C_d\sqrt p
\right)
\longrightarrow0.
\end{equation*}
The same convergence holds under $\mathbb P_p^{\mathrm{TC}}$ with $Z$ replaced by $Y-\mu$.
\end{lemma}

Let $L_p^{\mathrm{TC}}=d\mathbb P_p^{\mathrm{TC}}/d\mathbb Q_p^{\mathrm{TC}}$. The next calculation gives the low-degree part of the reverse gap.

\begin{proposition}[Low-degree likelihood-ratio bound]\label{prop_tensor_cut_ldlr}
Let
\[
R
=
\frac1p\sum_{i=1}^p\xi_i,
\]
where the $\xi_i$ are independent Rademacher variables. For every nonnegative integer $D$,
\begin{equation}\label{eq_tensor_cut_ldlr_exact}
\left\|(L_p^{\mathrm{TC}})_{\le D}\right\|_{L^2(\mathbb Q_p^{\mathrm{TC}})}^2
=
\sum_{m=0}^D
\frac{\lambda_p^{2m}}{m!}
\left(\E R^m\right)^d.
\end{equation}
Moreover, for every $m\ge1$, the summand indexed by $m$ is bounded in absolute value by
\begin{equation}\label{eq_tensor_cut_ldlr_term}
\left[
\frac{C_d\lambda_p^2m^{d/2-1}}{p^{d/2}}
\right]^m.
\end{equation}
Consequently, if
\begin{equation}\label{eq_tensor_cut_lambda}
\lambda_p
=
\frac{p^{d/4}}{(\log p)^B},
\qquad
B>\frac{d-2}{2},
\end{equation}
then, for every fixed $c>0$,
\begin{equation}\label{eq_tensor_cut_ldlr_bounded}
\left\|(L_p^{\mathrm{TC}})_{\le\lfloor c(\log p)^2\rfloor}\right\|_{L^2(\mathbb Q_p^{\mathrm{TC}})}
=
O(1).
\end{equation}
\end{proposition}

The proofs of Lemmas~\ref{lem_tensor_cut_norm},~\ref{lem_tensor_cut_signal}, and~\ref{lem_tensor_cut_noise}, together with the proof of Proposition~\ref{prop_tensor_cut_ldlr}, are collected in Section~\ref{sec_proofs_tensor_cut_reverse_gap}.

Similar to Conjecture \ref{conj_low_degree}, to interpret the preceding low-degree bound as evidence for polynomial-time hardness, we again use the standard low-degree heuristic in a model-specific form. This is the usual computational interpretation of low-degree likelihood-ratio bounds. 

\begin{conjecture}[Spiked-tensor low-degree conjecture]\label{conj_tensor_cut_low_degree}
Fix $d\ge3$ and $B>(d-2)/2$, and consider \eqref{eq_tensor_cut_model} with \eqref{eq_tensor_cut_lambda}. There exists $c_d>0$ such that boundedness of
\[
\left\|(L_p^{\mathrm{TC}})_{\le\lfloor c_d(\log p)^2\rfloor}\right\|_{L^2(\mathbb Q_p^{\mathrm{TC}})}
\]
implies that no sequence of tests generated by a single uniform polynomial-time algorithm (possibly randomized) strongly distinguishes $\mathbb P_p^{\mathrm{TC}}$ from $\mathbb Q_p^{\mathrm{TC}}$.
\end{conjecture}

Proposition~\ref{prop_tensor_cut_ldlr} verifies the antecedent of Conjecture~\ref{conj_tensor_cut_low_degree} at the signal strength \eqref{eq_tensor_cut_lambda}. Hence, under the conjecture, strong detection is computationally hard at this signal scale. On the other hand, Lemmas~\ref{lem_tensor_cut_signal} and~\ref{lem_tensor_cut_noise} show that the observation itself estimates the planted mean in tensor cut norm with signal separation $s_p=2^{-d}\lambda_p$ and estimation radius $r_p=C_d\sqrt p$. Consequently,
\begin{equation}\label{eq_tensor_cut_reverse_gap}
\frac{s_p}{r_p}
\asymp_d
\frac{\lambda_p}{\sqrt p}
=
\frac{p^{d/4-1/2}}{(\log p)^B}.
\end{equation}

\subsection{Computational lower bound}\label{subsec:tensor_cut_lower_bound}

We now apply the reduction of Section~\ref{sec:framework} to the reverse gap \eqref{eq_tensor_cut_reverse_gap}.
Let $f_p:\mathbb R^{p\times\cdots\times p}\to\mathbb R_+$ satisfy
\begin{equation}\label{eq_tensor_cut_approximation}
\rho_{f,p}^{-1}\|T\|_{\square,d}
\le
f_p(T)
\le
\zeta_{f,p}\|T\|_{\square,d},
\qquad
T\in\mathbb R^{p\times\cdots\times p},
\end{equation}
and set $\gamma_{f,p}=\rho_{f,p}\zeta_{f,p}$. The proof of the following theorem is collected in Section~\ref{sec_proof_thm_tensor_cut_lower_bound}.

\begin{theorem}[Conditional lower bound for tensor cut-norm approximation]\label{thm_tensor_cut_lower_bound}
Fix $d\ge3$ and $B>(d-2)/2$, and assume Conjecture~\ref{conj_tensor_cut_low_degree}. Let $(f_p)_{p\ge1}$ be a family generated by a single uniform deterministic polynomial-time algorithm and satisfying \eqref{eq_tensor_cut_approximation}. Assume the factors are efficiently computable and the ratios $2^{-d-1}\lambda_p/\rho_{f,p}$ are uniformly polynomial-time computable. Then
\begin{equation}\label{eq_tensor_cut_lower_bound}
\limsup_{p\to\infty}
\gamma_{f,p}
\frac{(\log p)^B}{p^{d/4-1/2}}
>
0.
\end{equation}
\end{theorem}

\begin{remark}[Randomized approximators]
The same conclusion extends to possibly randomized approximators that satisfy \eqref{eq_tensor_cut_approximation} with probability at least $2/3$ for every input. Indeed, for a sufficiently large constant $C$, repeat the algorithm independently $C\log p$ times and take the median before applying the reduction.
\end{remark}

\begin{remark}[Polynomial-threshold extension]\label{rem_tensor_cut_polynomial_threshold}
Under an affirmative answer to Question~4.2 of \cite{kunisky2019notes} for the spiked-tensor family, the same lower bound holds without a running-time assumption for every approximator of polynomial-threshold degree $O((\log p)^2)$. The precise statement appears in Section~\ref{sec:polynomial_threshold_approximators}.
\end{remark}

\begin{remark}[Comparison with known upper bounds]
    \label{remark:tensor_cut_upper_gap}

    For every fixed $d\ge2$, there is a randomized polynomial-time lower certificate $g_p$ and a constant $C_d>0$ such that
    \begin{equation*}\label{eq_tensor_cut_upper_summary}
    \frac1{C_dp^{(d-2)/2}}
    \|T\|_{\square,d}
    \le
    g_p(T)
    \le
    \|T\|_{\square,d}
    \end{equation*}
    for every order-$d$ tensor $T$, with probability at least $2/3$.
    The construction is based on the randomized polynomial-time approximation to the Boolean multilinear norm developed in \cite{heApproximationAlgorithmsDiscrete2013}.
    Proposition~\ref{prop_tensor_cut_upper_bound} gives the formal statement.
    
    For $d\ge3$, Theorem~\ref{thm_tensor_cut_lower_bound} and Proposition~\ref{prop_tensor_cut_upper_bound} leave a polynomial gap between
    \[
    \frac{p^{d/4-1/2}}{\polylog(p)}
    \qquad\text{and}\qquad
    p^{(d-2)/2}.
    \]
    When $d=2$, the reverse-gap ratio in \eqref{eq_tensor_cut_reverse_gap} does not diverge at the low-degree threshold. This agrees with the constant-factor polynomial-time approximation for the ordinary matrix cut norm obtained through Grothendieck's inequality in \cite{alon2006approximating}.

\end{remark}

\subsection{Connection with the literature}\label{subsec:tensor_cut_literature}

For comparison, define the normalized Boolean multilinear norm
\begin{equation*}\label{eq_boolean_multilinear_norm}
\|T\|_{\infty\to1,d}
:=
p^{-d/2}
\max_{x^{(1)},\ldots,x^{(d)}\in\{\pm1\}^p}
\left|
\left\langle
T,
x^{(1)}\otimes\cdots\otimes x^{(d)}
\right\rangle
\right|.
\end{equation*}
The tensor cut norm is standard and, for fixed $d$, is equivalent up to constants to this Boolean multilinear norm:
\begin{equation}\label{eq_tensor_cut_boolean_equivalence}
\|T\|_{\square,d}
\le
\|T\|_{\infty\to1,d}
\le
2^d\|T\|_{\square,d}.
\end{equation}
The first inequality follows by writing each indicator as $(\mathbf 1+x)/2$, and the second by partitioning each mode according to the signs of a Boolean vector. This equivalence and polynomial-time approximation algorithms for the tensor cut norm are developed in \cite{heApproximationAlgorithmsDiscrete2013}.

The low-degree calculation in Proposition~\ref{prop_tensor_cut_ldlr} is the asymmetric mode-wise analogue of the spiked-tensor calculation in \cite{kunisky2019notes}. The new approximation consequence comes from evaluating the additive Gaussian observation itself in the tensor cut norm: its stochastic radius is only $O_{\mathbb P}(\sqrt p)$, far below the conjectured computational detection scale $p^{d/4}/\polylog(p)$.

\begin{remark}[Why independent mode-wise subsets are needed]\label{rem_tensor_cut_symmetric_warning}
On the full symmetric tensor space, the functional
\[
T
\longmapsto
\max_{x\in\{\pm1\}^p}
|\langle T,x^{\otimes d}\rangle|
\]
need only be a seminorm. For example, the nonzero homogeneous polynomial $x_1^2x_2^{d-2}-x_2^d$ vanishes on the entire hypercube because $x_1^2=x_2^2=1$. Independent mode-wise subsets in \eqref{eq_tensor_cut_norm} avoid these hypercube identities, and singleton rectangles guarantee definiteness.
\end{remark}

\section{Discussion}\label{sec:discussion}

This work develops a general route from reverse detection--estimation gaps to computational lower bounds for norm approximation. Across the examples, the key quantity is the ratio between the norm separation under the alternative and the attainable estimation radius. When this ratio diverges, a sufficiently accurate norm approximator would convert an efficient estimator into a successful test. The resulting lower bounds are conditional on the corresponding planted-clique or low-degree hardness assumptions, as well as the factor and threshold computability conditions needed by the reduction.

Several questions concerning upper bounds remain open. For the sparse spectral norm and the normalized sparse matrix cut norm, the conditional lower bounds and the deterministic certificates differ only by a factor of order $\sqrt{\log p}$. For the tensor spectral and tensor cut norms, and consequently for the tensor nuclear norm through duality, the best general worst-case upper bounds remain polynomially larger than the conditional lower scales. Results for nonnegative and random tensors provide evidence that the lower exponent may be the correct one for tensor spectral norm, but no globally valid worst-case algorithm attaining this scale is currently known. Closing these gaps, either by improving the upper bounds or by strengthening the hardness constructions, is a central open problem.

The tensor nuclear-norm result obtained here is indirect and applies only to efficiently dualizable norm-valued approximators. We conjecture that a direct reverse detection--estimation gap also exists for the tensor nuclear norm itself. Constructing the required testing family, however, appears substantially more difficult. The alternative must have a large nuclear-norm signal while remaining computationally hard to detect, and an efficient estimator must simultaneously achieve a much smaller error in nuclear norm. The nuclear norm is highly sensitive to diffuse high-dimensional noise, whereas imposing sufficiently strong low-rank structure may make the alternative detectable through simpler matricization-based or Frobenius-type statistics. Balancing these competing requirements is an interesting open direction; a successful construction could remove the efficient-dualizability restriction and yield a direct computational lower bound for tensor nuclear-norm approximation.

More broadly, it would be useful to identify additional statistical models in which efficient estimation coexists with computationally hard detection in a nontrivial norm, and to understand when the reverse-gap ratio predicts the optimal polynomial-time approximation distortion.

\clearpage

\bibliographystyle{apalike}
\bibliography{reference}

\clearpage

\appendix

\section{Polynomial-Threshold Extensions}\label{sec:polynomial_threshold_approximators}

This section records the representation-class extensions described in Remarks~\ref{rem_polynomial_threshold_extension} and~\ref{rem_tensor_cut_polynomial_threshold}. Requiring the numerical output of a norm approximator itself to be a polynomial is incompatible with a global multiplicative norm guarantee except in trivial cases: a polynomial's behavior under scalar rescaling cannot agree globally with the degree-one positive homogeneity of a norm, and a nonzero linear form cannot be nonnegative on both $T$ and $-T$. The appropriate representation condition is instead imposed on the decision boundary used by the reduction.

\begin{definition}[Polynomial-threshold degree]\label{def_threshold_degree}
Let $\mathcal F_p$ be a finite-dimensional real vector space equipped with a fixed coordinate representation. A map $f_p:\mathcal F_p\to\mathbb R_+$ has \emph{polynomial-threshold degree at most $D$} if, for every $\tau\in\mathbb R$, there exists a real polynomial $q_{p,\tau}$ of total degree at most $D$ in the coordinates of $T\in\mathcal F_p$ such that
\begin{equation*}\label{eq_threshold_degree}
I_{\{f_p(T)>\tau\}}
=
I_{\{q_{p,\tau}(T)>0\}},
\qquad
T\in\mathcal F_p.
\end{equation*}
\end{definition}

The polynomial may vary with the threshold because the reduction fixes only one separating threshold at a time. Definition~\ref{def_threshold_degree} gives a nonempty class that is strictly larger than the class of polynomial-valued maps. For example, on any Euclidean coordinate space, the Frobenius norm $f_p(T)=\|T\|_F$ has polynomial-threshold degree at most two: for $\tau\ge0$, take $q_{p,\tau}(T)=\|T\|_F^2-\tau^2$, while for $\tau<0$ take the constant polynomial $1$.

\subsection{Tensor spectral norm}

The implication needed for the cumulant example is the bounded-LDLR direction posed as Question~4.2 in Section~4.1 of \cite{kunisky2019notes}. We state its specialization to the cumulant testing family.

\begin{assumption}[Bounded-LDLR thresholding for the cumulant family]\label{assump_bounded_ldlr_thresholding}
For the cumulant testing family in Section~\ref{sec:tensor}, let $M_p=O((\log p)^2)$ be any sequence of nonnegative integers such that
\[
\bigl\|(L_p)_{\le M_p}\bigr\|_{L^2(\mathbb Q_p)}
=
O(1).
\]
Then no sequence of sign-threshold tests $I_{\{q_p>0\}}$ strongly distinguishes $\mathbb P_p$ from $\mathbb Q_p$, where each $q_p$ is a real polynomial of total degree at most $M_p$ in the scalar coordinates of $(Y_1,\ldots,Y_n)$. Equivalently, every such sequence satisfies
\begin{equation}\label{eq_polynomial_threshold_strong_hardness}
\limsup_{p\to\infty}
\left[
\mathbb Q_p(q_p>0)+\mathbb P_p(q_p\le0)
\right]
>
0.
\end{equation}
\end{assumption}

\begin{theorem}[Polynomial-threshold tensor spectral-norm lower bound]\label{thm_polynomial_threshold_extension}
Fix $d\ge3$ and assume Assumption~\ref{assump_bounded_ldlr_thresholding}. Let $(f_p)_{p\ge1}$ be any sequence of maps $f_p:\operatorname{Sym}^d(\mathbb R^p)\to\mathbb R_+$ having polynomial-threshold degree at most $D_p$ and satisfying one of \eqref{eq_tensor_two_sided}, \eqref{eq_tensor_lower}, or \eqref{eq_tensor_upper}. There exist constants $c_d,B_d>0$, depending only on $d$, such that if $D_p\le c_d(\log p)^2$, then
\[
\limsup_{p\to\infty}
\gamma_{f,p}
\frac{(\log p)^{B_d}}{p^{d/4-1/2}}
>
0.
\]
For the choice \eqref{eq_choice_n_a}, one may take $B_d=9d/2$.
\end{theorem}

\begin{proof}
Consider the testing family, signal level $s_p$, estimator $\widehat{\bcK}_d$, and estimation radius $r_p$ constructed in Section~\ref{sec_proof}. Suppose, toward a contradiction, that $\gamma_{f,p}<s_p/(4r_p)$ for all sufficiently large $p$, and set $\tau_p=s_p/(2\rho_{f,p})$. The null and alternative inequalities in the proof of Theorem~\ref{thm_framework} give
\begin{equation}\label{eq_threshold_test_success}
\mathbb Q_p\bigl(f_p(\widehat{\bcK}_d)>\tau_p\bigr)
+
\mathbb P_p\bigl(f_p(\widehat{\bcK}_d)\le\tau_p\bigr)
\longrightarrow0.
\end{equation}

By Definition~\ref{def_threshold_degree}, there is a polynomial $q_{p,\tau_p}$ of degree at most $D_p$ in the entries of its tensor argument such that
\begin{equation*}\label{eq_composed_threshold}
I_{\{f_p(\widehat{\bcK}_d)>\tau_p\}}
=
I_{\{q_{p,\tau_p}(\widehat{\bcK}_d)>0\}}.
\end{equation*}
Every coordinate of $\widehat{\bcK}_d$ has degree at most $d$ in the observed samples by \eqref{eq_sample_cumulant_partition}. Hence $q_{p,\tau_p}(\widehat{\bcK}_d)$ has degree at most $M_p=dD_p=O((\log p)^2)$ in the samples. In particular, $M_p=o(\min\{p,n\})$, and the calculation following \eqref{eq_ldlr_bound}, which is valid for every degree cutoff of order $(\log p)^2$, gives
\[
\bigl\|(L_p)_{\le M_p}\bigr\|_{L^2(\mathbb Q_p)}
=
O(1).
\]
Equation~\eqref{eq_threshold_test_success} contradicts \eqref{eq_polynomial_threshold_strong_hardness}. Therefore
\[
\limsup_{p\to\infty}
\frac{\gamma_{f,p}r_p}{s_p}
\ge
\frac14.
\]
Combining this inequality with \eqref{eq_signal_scale}, \eqref{eq_estimation_scale_cumulant}, and $n\asymp p^{d-1}$ gives
\[
\limsup_{p\to\infty}
\gamma_{f,p}
\frac{(\log p)^{9d/2}}{p^{d/4-1/2}}
>
0.
\]
\end{proof}

\subsection{Tensor cut norm}

We use the analogous specialization of Question~4.2 for the asymmetric spiked-tensor family.

\begin{assumption}[Bounded-LDLR thresholding for the spiked-tensor family]\label{assump_tensor_cut_bounded_ldlr_thresholding}
Fix $d\ge3$ and $B>(d-2)/2$, and consider \eqref{eq_tensor_cut_model} with \eqref{eq_tensor_cut_lambda}. Let $M_p=O((\log p)^2)$ be any sequence for which
\[
\bigl\|(L_p^{\mathrm{TC}})_{\le M_p}\bigr\|_{L^2(\mathbb Q_p^{\mathrm{TC}})}
=
O(1).
\]
Then no sequence of sign-threshold tests $I_{\{q_p(Y)>0\}}$ strongly distinguishes $\mathbb P_p^{\mathrm{TC}}$ from $\mathbb Q_p^{\mathrm{TC}}$, where each $q_p$ is a real polynomial of total degree at most $M_p$ in the entries of $Y$.
\end{assumption}

\begin{theorem}[Polynomial-threshold tensor cut-norm lower bound]\label{thm_tensor_cut_polynomial_threshold}
Fix $d\ge3$ and $B>(d-2)/2$, and assume Assumption~\ref{assump_tensor_cut_bounded_ldlr_thresholding}. Let $(f_p)_{p\ge1}$ be a sequence of maps $f_p:\mathbb R^{p\times\cdots\times p}\to\mathbb R_+$ having polynomial-threshold degree at most $D_p$ and satisfying \eqref{eq_tensor_cut_approximation}. There is a constant $c_d>0$ such that, if $D_p\le c_d(\log p)^2$, then
\[
\limsup_{p\to\infty}
\gamma_{f,p}
\frac{(\log p)^B}{p^{d/4-1/2}}
>
0.
\]
\end{theorem}

\begin{proof}
Suppose, toward a contradiction, that $\gamma_{f,p}<s_p/(4r_p)$ for all sufficiently large $p$, where $s_p=2^{-d}\lambda_p$ and $r_p=C_d\sqrt p$ are as in Section~\ref{subsec:tensor_cut_reverse_gap}. Set $\tau_p=s_p/(2\rho_{f,p})$. The proof of Theorem~\ref{thm_framework}, applied with $\widehat\mu=Y$, gives
\[
\mathbb Q_p^{\mathrm{TC}}\bigl(f_p(Y)>\tau_p\bigr)
+
\mathbb P_p^{\mathrm{TC}}\bigl(f_p(Y)\le\tau_p\bigr)
\longrightarrow0.
\]
By Definition~\ref{def_threshold_degree}, this test is the sign threshold of a polynomial $q_{p,\tau_p}(Y)$ of degree at most $D_p$. Proposition~\ref{prop_tensor_cut_ldlr} gives
\[
\bigl\|(L_p^{\mathrm{TC}})_{\le D_p}\bigr\|_{L^2(\mathbb Q_p^{\mathrm{TC}})}
=
O(1),
\]
contradicting Assumption~\ref{assump_tensor_cut_bounded_ldlr_thresholding}. Therefore
\[
\limsup_{p\to\infty}
\frac{\gamma_{f,p}r_p}{s_p}
\ge
\frac14,
\]
and \eqref{eq_tensor_cut_reverse_gap} proves the result.
\end{proof}

\section{A Quantitative Refinement of the Koiran--Zouzias Reduction}\label{sec_kz_quantitative_refinement}

Theorems~3--5 of \cite{koiranHiddenCliquesCertification2014} establish hardness for polynomially separated RIP parameters. We extract a complementary consequence of their Cholesky reduction while keeping the RIP order equal to the clique size. Fix $\varepsilon\in(0,1/2)$ and let $k=\lceil p^{1/2-\varepsilon}\rceil$.
Hypothesis $(H_\varepsilon)$ of \cite{koiranHiddenCliquesCertification2014} asserts that no deterministic polynomial-time algorithm accepts every graph on $p$ vertices containing a $k$-clique while rejecting a graph drawn from $G(p,1/2)$ with probability tending to one.
We optimize the null parameter in their Proposition~4 against the clique obstruction in their Proposition~5, obtaining a lower bound that rules out distortion $o(\sqrt{k/\log(ep/k)})$.

Given a graph $G$, let $W$ be its signed adjacency matrix, with zero diagonal, $W_{ij}=1$ when $\{i,j\}$ is an edge, and $W_{ij}=-1$ otherwise.
Fix a constant $c_0\in(0,1/3)$ and form
\[
B=I+\frac{c_0}{\sqrt p}W.
\]
If $B$ is positive semidefinite, the Cholesky reduction returns a matrix $X$ satisfying $X^\top X=B$; otherwise, it returns $X=0$.

\begin{proposition}\label{prop_kz_quantitative_refinement}
Under Hypothesis $(H_\varepsilon)$ of \cite{koiranHiddenCliquesCertification2014}, let $(f_{p,k})_{p\ge1}$ be a family generated by a single uniform deterministic polynomial-time algorithm and satisfying
\[
\rho_{f,p}^{-1}\delta_k(X)\le f_{p,k}(X)\le\zeta_{f,p}\delta_k(X)
\]
for every input $X$, with $\gamma_{f,p}=\rho_{f,p}\zeta_{f,p}$. Assume that the approximation factors are efficiently computable and that the ratios
\[
\tau_p=\frac{c_0(k-1)}{2\rho_{f,p}\sqrt p}
\]
are uniformly polynomial-time computable. Then
\[
\limsup_{p\to\infty}
\gamma_{f,p}\sqrt{\frac{\log(ep/k)}{k}}
>0.
\]
\end{proposition}

\begin{remark}
Because $\delta_k(X)=\|X^\top X-I\|_{\mathrm{sp},k}$, any deterministic global approximator for the sparse spectral norm can be applied to $X^\top X-I$.
The same lower bound, ruling out $o(\sqrt{k/\log(ep/k)})$ distortion, therefore holds for deterministic global sparse spectral norm approximation under the factor and threshold computability conditions of Proposition~\ref{prop_kz_quantitative_refinement}.
\end{remark}

\begin{proof}
Choose a sufficiently large universal constant $C_0$ and define
\begin{equation*}\label{eq_kz_null_parameter}
\delta_0=C_0\sqrt{\frac{k\log(ep/k)}{p}}.
\end{equation*}
Proposition~4 of \cite{koiranHiddenCliquesCertification2014} shows that, when $G\sim G(p,1/2)$, the matrix $X$ satisfies the RIP of order $k$ with parameter $\delta_0$ with probability at least
\begin{equation}\label{eq_kz_null_probability}
1-4\exp\left\{k\log\left(\frac{ep}{k}\right)-\frac{1}{32}\left(\frac{\delta_0\sqrt p}{c_0}-3\sqrt{k}\right)^2\right\}-4\exp\left\{-\frac{(1/c_0-3)^2p}{32}\right\}.
\end{equation}
The condition required in that proposition is
\[
3c_0<\min\left\{1,\frac{\delta_0\sqrt p}{\sqrt{k}}\right\},
\]
which holds for all sufficiently large $p$ because $c_0<1/3$ and
\[
\frac{\delta_0\sqrt p}{\sqrt{k}}=C_0\sqrt{\log\left(\frac{ep}{k}\right)}.
\]
Moreover,
\[
\left(\frac{\delta_0\sqrt p}{c_0}-3\sqrt{k}\right)^2=k\left(\frac{C_0}{c_0}\sqrt{\log\left(\frac{ep}{k}\right)}-3\right)^2.
\]
Taking $C_0$ sufficiently large makes this quantity at least $64k\log(ep/k)$.
The lower bound in \eqref{eq_kz_null_probability} therefore tends to one, and hence
\begin{equation}\label{eq_kz_null_bound}
\delta_k(X)\le C_0\sqrt{\frac{k\log(ep/k)}{p}}
\end{equation}
with probability tending to one under $G(p,1/2)$.

Suppose next that $G$ contains a clique $S$ of size $k$.
If $B$ is positive semidefinite, take $u=\mathbf 1_S/\sqrt{k}$.
Since $u^\top Wu=k-1$, we have
\[
\|Xu\|_2^2-1=u^\top(B-I)u=\frac{c_0(k-1)}{\sqrt p}.
\]
If $B$ is not positive semidefinite, then $X=0$ and $\delta_k(X)=1$, which is larger than $c_0(k-1)/\sqrt p$ for all sufficiently large $p$.
Thus Proposition~5 of \cite{koiranHiddenCliquesCertification2014} gives, for every graph containing a $k$-clique,
\begin{equation}\label{eq_kz_clique_bound}
\delta_k(X)\ge\frac{c_0(k-1)}{\sqrt p}.
\end{equation}

Combining \eqref{eq_kz_null_bound} and \eqref{eq_kz_clique_bound}, the ratio between the clique and null cases is
\[
\frac{c_0(k-1)/\sqrt p}{C_0\sqrt{k\log(ep/k)/p}}\asymp\sqrt{\frac{k}{\log(ep/k)}}.
\]
Suppose, toward a contradiction, that $\gamma_{f,p}=o\bigl(\sqrt{k/\log(ep/k)}\bigr)$. Then, for all sufficiently large $p$,
\[
\zeta_{f,p}\delta_0
=
\frac{\gamma_{f,p}}{\rho_{f,p}}C_0\sqrt{\frac{k\log(ep/k)}p}
<
\frac{c_0(k-1)}{2\rho_{f,p}\sqrt p}
=
\tau_p.
\]
Under $G(p,1/2)$, the output is therefore below $\tau_p$ with probability tending to one, while on every graph containing a $k$-clique it is at least $\rho_{f,p}^{-1}c_0(k-1)/\sqrt p>\tau_p$. By the assumed uniform computability of $\tau_p$, thresholding at this value gives a deterministic polynomial-time algorithm that accepts every graph containing a $k$-clique and rejects a graph drawn from $G(p,1/2)$ with probability tending to one. This contradicts Hypothesis $(H_\varepsilon)$ and proves the claimed lower bound.
\end{proof}

\section{Polynomial-Time Upper Bounds for the Sparse Matrix Norms}\label{sec_sparse_spectral_upper_bound}

This section constructs deterministic polynomial-time lower certificates with distortion $\sqrt{k}$ for the sparse spectral norm and the normalized sparse matrix cut norm.

\begin{proposition}\label{prop_sparse_spectral_upper_bound}
There is a deterministic polynomial-time algorithm $f_{p,k}$ such that, for every symmetric matrix $M\in\mathbb R^{p\times p}$,
\[
\frac{1}{\sqrt{k}}\|M\|_{\mathrm{sp},k}\le f_{p,k}(M)\le\|M\|_{\mathrm{sp},k}.
\]
\end{proposition}

\begin{proof}
Let $u_\star$ attain $\|M\|_{\mathrm{sp},k}$.
Since $\|u_\star\|_1\le\sqrt{k}$,
\[
\|M\|_{\mathrm{sp},k}=|u_\star^\top M u_\star|\le\|u_\star\|_1\max_{i\in\operatorname{supp}(u_\star)}|e_i^\top M u_\star|.
\]
Consequently, there exists $i\in\operatorname{supp}(u_\star)$ such that
\begin{equation}\label{eq_sparse_upper_cross_term}
|e_i^\top M u_\star|\ge\frac{\|M\|_{\mathrm{sp},k}}{\sqrt{k}}.
\end{equation}

For every $i\in[p]$, let $T_i$ contain $i$ and the indices of the $k-1$ largest off-diagonal entries in absolute value in the $i$th row.
If $\sum_{j\in T_i}M_{ij}^2>0$, define
\[
x_i=\frac{\sum_{j\in T_i}M_{ij}e_j}{\left(\sum_{j\in T_i}M_{ij}^2\right)^{1/2}};
\]
otherwise, set $x_i=e_i$.
The vector $x_i$ is a unit vector supported on $T_i$.
For the index $i$ identified in \eqref{eq_sparse_upper_cross_term}, the choice of $T_i$ and the Cauchy--Schwarz inequality give
\[
|e_i^\top Mx_i|=\left(\sum_{j\in T_i}M_{ij}^2\right)^{1/2}\ge\left(\sum_{j\in\operatorname{supp}(u_\star)}M_{ij}^2\right)^{1/2}\ge|e_i^\top M u_\star|.
\]
Therefore,
\begin{equation}\label{eq_sparse_upper_selected_cross_term}
|e_i^\top Mx_i|\ge\frac{\|M\|_{\mathrm{sp},k}}{\sqrt{k}}.
\end{equation}

For every $i$ and each sign for which $e_i\pm x_i\neq0$, define
\[
y_{i,\pm}=\frac{e_i\pm x_i}{\|e_i\pm x_i\|_2}.
\]
Every defined candidate is supported on $T_i$ and hence is $k$-sparse.
The polarization identity gives
\[
(e_i+x_i)^\top M(e_i+x_i)-(e_i-x_i)^\top M(e_i-x_i)=4e_i^\top Mx_i,
\]
while
\[
\|e_i+x_i\|_2^2+\|e_i-x_i\|_2^2=4.
\]
If both $e_i+x_i$ and $e_i-x_i$ are nonzero, the polarization identity and the squared-norm identity above imply that at least one of $y_{i,+}$ and $y_{i,-}$ satisfies
\[
|y_{i,\pm}^\top M y_{i,\pm}|\ge|e_i^\top Mx_i|.
\]
Indeed, if both normalized quadratic forms were smaller than $|e_i^\top Mx_i|$, the triangle inequality and $\|e_i+x_i\|_2^2+\|e_i-x_i\|_2^2=4$ would contradict the polarization identity. If one of $e_i+x_i$ and $e_i-x_i$ is zero, the other normalized candidate has quadratic-form magnitude exactly $|e_i^\top Mx_i|$. Thus at least one defined candidate satisfies the displayed inequality.

Define $f_{p,k}(M)$ as the largest absolute quadratic form obtained from all nonzero candidates $y_{i,+}$ and $y_{i,-}$.
Every candidate is feasible for \eqref{eq_sparse_spectral_norm}, so
\[
f_{p,k}(M)\le\|M\|_{\mathrm{sp},k}.
\]
For the index satisfying \eqref{eq_sparse_upper_selected_cross_term}, at least one defined candidate gives
\[
f_{p,k}(M)\ge\frac{\|M\|_{\mathrm{sp},k}}{\sqrt{k}}.
\]
Constructing the sets $T_i$ and evaluating the resulting candidates takes polynomial time, which completes the proof.
\end{proof}

\begin{proposition}\label{prop_sparse_cut_upper_bound}
There is a deterministic polynomial-time algorithm $g_{p,k}$ such that, for every real matrix $M\in\mathbb R^{p\times p}$,
\[
\frac{1}{\sqrt{k}}\|M\|_{\square,k}
\le
g_{p,k}(M)
\le
\|M\|_{\square,k}.
\]
\end{proposition}

\begin{proof}
For each row $i\in[p]$ and each cardinality $1\le c\le k$, let $C_{i,c}^{+}$ index the $c$ largest entries in the $i$th row, and let $C_{i,c}^{-}$ index the $c$ smallest entries. Define
\[
g_{p,k}(M)
=
\max_{\substack{i\in[p],\,1\le c\le k\\ \sigma\in\{+,-\}}}
\frac{\left|\sum_{j\in C_{i,c}^{\sigma}}M_{ij}\right|}{\sqrt c}.
\]
Each candidate is feasible for \eqref{eq_sparse_cut_norm} with $R=\{i\}$, so $g_{p,k}(M)\le\|M\|_{\square,k}$.

Let $(R_\star,C_\star)$ attain $\|M\|_{\square,k}$ and set $r=|R_\star|$ and $c=|C_\star|$. By the triangle inequality, some $i_\star\in R_\star$ satisfies
\[
\left|
\sum_{j\in C_\star}M_{i_\star j}
\right|
\ge
\frac{
\left|
\sum_{i\in R_\star,j\in C_\star}M_{ij}
\right|
}{r}.
\]
For the fixed row $i_\star$ and cardinality $c$, either the sum of the $c$ largest entries or the sum of the $c$ smallest entries has absolute value at least the absolute sum over $C_\star$. Consequently,
\[
g_{p,k}(M)
\ge
\frac1{\sqrt c}
\left|
\sum_{j\in C_\star}M_{i_\star j}
\right|
\ge
\frac{
\left|
\sum_{i\in R_\star,j\in C_\star}M_{ij}
\right|
}{r\sqrt c}
=
\frac{\|M\|_{\square,k}}{\sqrt r}
\ge
\frac{\|M\|_{\square,k}}{\sqrt k}.
\]
Sorting each row and evaluating all prefixes of lengths at most $k$ takes polynomial time.
\end{proof}

\section{Polynomial-Time Upper Bound for the Tensor Cut Norm}\label{sec_tensor_cut_upper_bound}

This section records the randomized polynomial-time lower certificate summarized in Remark~\ref{remark:tensor_cut_upper_gap}.

\begin{proposition}[Known polynomial-time upper bound]\label{prop_tensor_cut_upper_bound}
For every fixed $d\ge2$, there is a randomized polynomial-time lower-certificate algorithm $g_p$ and a constant $C_d>0$ such that, for every order-$d$ tensor $T$, the inequality
\begin{equation}\label{eq_tensor_cut_upper_bound}
\frac1{C_dp^{(d-2)/2}}
\|T\|_{\square,d}
\le
g_p(T)
\le
\|T\|_{\square,d}
\end{equation}
holds with probability at least $2/3$.
\end{proposition}

\begin{proof}
Theorem~3.1 and Proposition~4.1 of \cite{heApproximationAlgorithmsDiscrete2013} give a randomized polynomial-time approximation to the Boolean multilinear norm with ratio $\Omega_d(p^{-(d-2)/2})$ in equal dimensions. Combining this result with \eqref{eq_tensor_cut_boolean_equivalence}, we enumerate the $2^d$ sign rectangles determined by the returned sign vectors and output the largest absolute normalized rectangle sum. At least one rectangle has absolute value at least $2^{-d}$ times the returned Boolean value. Since $d$ is fixed, this procedure runs in polynomial time and gives \eqref{eq_tensor_cut_upper_bound}. Repeating independently and taking the largest returned certificate raises the success probability to at least $2/3$ without changing the distortion order.
\end{proof}

\section{Additional lemmas}

\begin{lemma}\label{lemma_existence_1d}
Fix an integer $d\ge3$.
There exists a real-valued sub-Gaussian random variable $X$ whose first $d-1$ cumulants agree with those of a standard Gaussian random variable, while its $d$th cumulant is nonzero.
Equivalently,
\[
\kappa_1(X)=0,\qquad
\kappa_2(X)=1,\qquad
\kappa_j(X)=0\quad\text{for }3\le j\le d-1,
\qquad
\kappa_d(X)\ne0.
\]
Moreover, there exists a constant $C_d\ge1$, depending only on $d$, such that
\[
\mathbb E|H_m(X)|
\le
m!C_d^m,
\qquad
m\ge1,
\]
where $H_m$ denotes the $m$th probabilists' Hermite polynomial.
\end{lemma}

\begin{proof}
Let $Z\sim\mathcal N(0,1)$, let $\phi$ denote its density, and set $H_0\equiv1$.
Choose $d+1$ distinct points $x_0,\ldots,x_d\in\mathbb R$.
Because $H_i$ is a monic polynomial of degree $i$, the matrix
\[
\left(H_i(x_j)\right)_{0\le i,j\le d}
\]
is invertible; indeed, its determinant is the Vandermonde determinant
\[
\prod_{0\le j<k\le d}(x_k-x_j)\ne0.
\]

For sufficiently small $\eta>0$, let
\[
I_j=(x_j-\eta,x_j+\eta),
\qquad
j=0,\ldots,d,
\]
so that the intervals are pairwise disjoint.
Define the $(d+1)\times(d+1)$ matrix $A=(a_{ij})_{0\le i,j\le d}$ by
\[
a_{ij}
=
\mathbb E\left[H_i(Z)\mathbf 1_{\{Z\in I_j\}}\right].
\]
Writing $p_j=\mathbb P(Z\in I_j)>0$, we have
\[
\frac{a_{ij}}{p_j}
=
\mathbb E\left[H_i(Z)\mid Z\in I_j\right]
\longrightarrow
H_i(x_j)
\]
as $\eta\downarrow0$.
Thus, after decreasing $\eta$ if necessary, the matrix $A$ is invertible.

The first $d$ rows of $A$ therefore have rank $d$, so there exists a nonzero vector $c=(c_0,\ldots,c_d)^\top$ such that
\[
\sum_{j=0}^d c_j a_{ij}=0,
\qquad
i=0,\ldots,d-1.
\]
Since $A$ is invertible, this vector necessarily satisfies
\[
\sum_{j=0}^d c_j a_{dj}\ne0.
\]
After rescaling, assume that $\max_{0\le j\le d}|c_j|=1$, and define
\[
h(x)
=
\sum_{j=0}^d c_j\mathbf 1_{\{x\in I_j\}}.
\]
Because the intervals are pairwise disjoint, $\|h\|_\infty=1$, and the preceding identities give
\[
\mathbb E\left[H_i(Z)h(Z)\right]=0,
\qquad
i=0,\ldots,d-1,
\]
whereas
\[
\mathbb E\left[H_d(Z)h(Z)\right]\ne0.
\]

Define $X$ by the density
\[
f_X(x)
=
\left(1+\frac12h(x)\right)\phi(x).
\]
The identity $\mathbb E h(Z)=0$, corresponding to $i=0$, shows that $f_X$ integrates to one.
Moreover, since $\|h\|_\infty=1$,
\[
\frac12\phi(x)
\le
f_X(x)
\le
\frac32\phi(x),
\qquad
x\in\mathbb R.
\]
In particular,
\[
\mathbb E\exp\left(\frac{X^2}{9}\right)
\le
\frac32\mathbb E\exp\left(\frac{Z^2}{9}\right)
=
\frac{9}{2\sqrt7}
<
2,
\]
so $X$ is sub-Gaussian.

For every $1\le i\le d-1$, the orthogonality relations above and the Gaussian Hermite identities give
\[
\mathbb E H_i(X)
=
\mathbb E H_i(Z)
+
\frac12\mathbb E\left[H_i(Z)h(Z)\right]
=
0.
\]
On the other hand,
\[
\mathbb E H_d(X)
=
\frac12\mathbb E\left[H_d(Z)h(Z)\right]
\ne
0.
\]
Because $H_0,\ldots,H_r$ form a basis of the polynomials of degree at most $r$, the identities $\mathbb E H_i(X)=\mathbb E H_i(Z)$ for $0\le i\le d-1$ imply that the moments of $X$ and $Z$ agree through order $d-1$.
Consequently, the first $d-1$ cumulants of $X$ agree with the Gaussian cumulants.
The moment--cumulant relation then gives
\[
\kappa_d(X)
=
\mathbb E H_d(X)
\ne
0.
\]

Finally, the density bound and the Hermite orthogonality relation $\mathbb E H_m(Z)^2=m!$ imply
\[
\mathbb E|H_m(X)|
\le
\frac32\mathbb E|H_m(Z)|
\le
\frac32\left(\mathbb E H_m(Z)^2\right)^{1/2}
=
\frac32\sqrt{m!}.
\]
For every $m\ge1$,
\[
\frac32\sqrt{m!}
\le
m!\left(\frac32\right)^m.
\]
Thus the asserted bound holds, for example, with $C_d=3/2$.
\end{proof}

\section{Proofs}

\subsection{The cumulant reverse gap and proof of Theorem~\ref{thm_main}}\label{sec_proof}

\begin{proof}[Proof of Proposition~\ref{prop_cumulant_reverse_gap} and Theorem~\ref{thm_main}]
For the testing family defined in Section~\ref{sec:tensor}, Lemma~1 of \cite{tang2026detectionharderestimationcertain} gives, for $M=o(\min\{p,n\})$,
\begin{equation}\label{eq_ldlr_bound}
\bigl\|(L_p)_{\le M}\bigr\|_{L^2(\mathbb Q_p)}^2
\le
\sum_{0\le m\le M}
\left(
\frac{a}{1+a}\frac{C_d m^4n^{1/d}}{p^{1/2}}
\right)^m.
\end{equation}
Fix an arbitrary constant $c>0$ and let $M=\lfloor c(\log p)^2\rfloor$. Then $M=o(\min\{p,n\})$ and, uniformly over $m\le M$,
\[
\frac{a}{1+a}\frac{C_d m^4n^{1/d}}{p^{1/2}}
\le
\frac{C_d m^4}{(\log p)^9}
\le
C_{d,c}' (\log p)^{-1}
=o(1).
\]
Consequently, the sum in \eqref{eq_ldlr_bound} is $O(1)$ for every fixed $c>0$, which proves the low-degree claim of Proposition~\ref{prop_cumulant_reverse_gap}. Taking $c=c_d$, the constant in Conjecture~\ref{conj_low_degree}, shows that \eqref{eq_tensor_ldlr_conjecture_antecedent} holds. Conjecture~\ref{conj_low_degree} therefore gives the ordinary strong-detection hardness condition \eqref{eq_tensor_polytime_strong_hardness} for this testing family.

We next record the separation and the estimation radius. Conditional on $U$, the $d$th cumulant tensor under the alternative is
\[
\bcK_d(Y_i\mid U)
=
\kappa_d(W_1)\left(\frac{a}{p}\right)^{d/2}(S_UU)^{\otimes d},
\]
where $\kappa_d(W_1)\ne0$ depends only on $d$. Since $\|S_UU\|_2^2=p/(1+a)$, the signal norm is
\begin{equation}\label{eq_signal_scale}
s_p
=
\|\bcK_d(Y_i\mid U)\|_s
=
|\kappa_d(W_1)|\left(\frac{a}{1+a}\right)^{d/2}
\asymp_d
\frac{p^{d/4}}{n^{1/2}(\log p)^{9d/2}}.
\end{equation}
This value does not depend on the realization of $U$. We take $\mu_p=0$ under $\mathbb Q_p$ and $\mu_p=\bcK_d(Y_i\mid U)$ under $\mathbb P_p$, so \eqref{eq_general_separation} holds with the value $s_p$ in \eqref{eq_signal_scale}.

Let $\widehat{\bcM}_j=n^{-1}\sum_{i=1}^nY_i^{\otimes j}$ and let $\widehat{\bcK}_d$ be the plug-in sample cumulant. The moment--cumulant formula expresses $\widehat{\bcK}_d$ as a finite sum over set partitions $\pi$ of $[d]$:
\begin{equation}\label{eq_sample_cumulant_partition}
\widehat{\bcK}_d
=
\sum_{\pi\in\Pi_d}
(|\pi|-1)!(-1)^{|\pi|-1}
\bigotimes_{B\in\pi}\widehat{\bcM}_{|B|},
\end{equation}
with the tensor indices assigned according to the blocks of $\pi$ and $\Pi_d$ denoting the set of all set partitions of $[d]$. For every partition, $\sum_{B\in\pi}|B|=d$. Hence every coordinate of $\widehat{\bcK}_d$ is a polynomial of total degree at most $d$ in the observed samples.
Because $d$ is fixed, \eqref{eq_sample_cumulant_partition} also shows that $\widehat{\bcK}_d$ is computable in polynomial time.

Conditional on $U$, the observations under $\mathbb P_p$ are independent, and the sub-Gaussian constants needed below are uniform in $U$ because $a=o(1)$, $d$ is fixed, and $\|S_U\|_{\mathrm{op}}\le1$. Theorem~4 and the cumulant discussion following it in \cite{tang2026detectionharderestimationcertain} apply under $\mathbb Q_p$ and, conditionally on $U$, under $\mathbb P_p$. They give the following bound with probability tending to one; under $\mathbb P_p$, the convergence is uniform in $U$:
\[
\|\widehat{\bcK}_d-\mu_p\|_s
\le
C_d\max\left\{\frac{p^{d/2}}{n},\sqrt{\frac pn}\right\}.
\]
Because $n<p^{d-1}$, the first term dominates. Thus we may take
\begin{equation}\label{eq_estimation_scale_cumulant}
r_p=C_d\frac{p^{d/2}}{n}
\end{equation}
in \eqref{eq_estimation_radius}, after integrating the conditional probability bound over $U$ under $\mathbb P_p$.

Equations~\eqref{eq_ldlr_bound}, \eqref{eq_signal_scale}, and~\eqref{eq_estimation_scale_cumulant} prove Proposition~\ref{prop_cumulant_reverse_gap}. We now apply the general reduction.

Let $f_p$ be as in Theorem~\ref{thm_main}.
Since $\widehat{\bcK}_d$ and $f_p$ are uniformly polynomial-time computable and the ratios $\tau_p:=s_p/(2\rho_{f,p})$ are uniformly polynomial-time computable by hypothesis, the tests $\phi_p=I_{\{f_p(\widehat{\bcK}_d)>\tau_p\}}$ are uniform polynomial-time tests.

If $\gamma_{f,p}<s_p/(4r_p)$ for all sufficiently large $p$, the null and alternative inequalities in the proof of Theorem~\ref{thm_framework} give
\[
\mathbb Q_p(\phi_p=1)+\mathbb P_p(\phi_p=0)\longrightarrow0.
\]
This contradicts \eqref{eq_tensor_polytime_strong_hardness}. Hence
\[
\limsup_{p\to\infty}\frac{\gamma_{f,p}r_p}{s_p}\ge\frac14.
\]
Combining this inequality with \eqref{eq_signal_scale}, \eqref{eq_estimation_scale_cumulant}, and $n\asymp p^{d-1}$ from \eqref{eq_choice_n_a} gives
\[
\limsup_{p\to\infty}
\gamma_{f,p}\frac{(\log p)^{9d/2}}{p^{d/4-1/2}}>0.
\]
This completes the proof.
\end{proof}

\subsection{Proof of Corollary \ref{cor_tensor_nuclear_lower_bound}}\label{sec:tensor_nuclear_duality}

\begin{proof}
Set $f_p=g_p^\circ$. The inequalities in \eqref{eq_tensor_nuclear_approximation} imply the unit-ball inclusions
\[
\left\{\bcT:\|\bcT\|_*\le\frac1{\zeta_{g,p}}\right\}
\subseteq
\left\{\bcT:g_p(\bcT)\le1\right\}
\subseteq
\left\{\bcT:\|\bcT\|_*\le\rho_{g,p}\right\}.
\]
Taking support functions and using \eqref{eq_tensor_spectral_nuclear_duality} gives, for every $\bcS\in\operatorname{Sym}^d(\mathbb R^p)$,
\begin{equation*}
\frac1{\zeta_{g,p}}\|\bcS\|_s
\le
f_p(\bcS)
\le
\rho_{g,p}\|\bcS\|_s.
\end{equation*}
Thus $f_p$ approximates the tensor spectral norm with factors
\[
\rho_{f,p}=\zeta_{g,p},
\qquad
\zeta_{f,p}=\rho_{g,p},
\qquad
\gamma_{f,p}=\gamma_{g,p}.
\]
By efficient dualizability, $(f_p)_{p\ge1}$ is generated by a single uniform deterministic polynomial-time algorithm. Its separating threshold in Theorem~\ref{thm_main} is
\[
\frac{s_p^{\mathrm{cum}}}{2\rho_{f,p}}
=
\frac{s_p^{\mathrm{cum}}}{2\zeta_{g,p}},
\]
which is uniformly polynomial-time computable by assumption. Applying Theorem~\ref{thm_main} proves \eqref{eq_tensor_nuclear_lower_bound}.
\end{proof}

\subsection{Proof of Lemma~\ref{lem_sparse_spectral_norm}}\label{sec_proof_lem_sparse_spectral_norm}

\begin{proof}
Absolute homogeneity and the triangle inequality follow directly from \eqref{eq_sparse_spectral_norm}.
It remains to verify definiteness.
Suppose that $\|M\|_{\mathrm{sp},k}=0$.
Taking $u=e_i$ gives $M_{ii}=0$ for every $i\in[p]$.
For distinct $i,j\in[p]$, the vector $u=(e_i+e_j)/\sqrt2$ is two-sparse, so
\[
0=u^\top M u=\frac{M_{ii}+M_{jj}+2M_{ij}}{2}=M_{ij}.
\]
Thus every entry of $M$ is zero.
\end{proof}

\subsection{Proof of Lemma~\ref{lem_sparse_cut_norm}}\label{sec_proof_lem_sparse_cut_norm}

\begin{proof}
Absolute homogeneity and the triangle inequality follow directly from \eqref{eq_sparse_cut_norm}. If $\|M\|_{\square,k}=0$, then taking $R=\{i\}$ and $C=\{j\}$ gives $|M_{ij}|=0$ for every $i,j\in[p]$. Hence $M=0$.
\end{proof}

\subsection{Proof of Theorem \ref{thm_sparse_norm_concentration}}\label{sec_proof_thm_sparse_norm_concentration}

We prove for the two norms separately. 

\begin{lemma}[Concentration of the sparse spectral norm]\label{lem_sparse_spectral_concentration}
There exist universal constants $C,c>0$ such that, under $\mathbb Q_p^{\mathrm{PC}}$ and under $\mathbb P_p^{\mathrm{PC}}$ conditional on any planted set $S$,
\begin{equation}\label{eq_sparse_spectral_concentration}
\mathbb P\left(
\|\widehat\mu-\mu\|_{\mathrm{sp},k}
>
C\sqrt{k\log\left(\frac{ep}{k}\right)}
\right)
\le
2\exp\left(
-ck\log\left(\frac{ep}{k}\right)
\right).
\end{equation}
Here $\mu=0$ under the null and $\mu=\mu_S$ under the alternative.
\end{lemma}

\begin{lemma}[Concentration of the sparse cut norm]\label{lem_sparse_cut_concentration}
There exist universal constants $C,c>0$ such that, under $\mathbb Q_p^{\mathrm{PC}}$ and under $\mathbb P_p^{\mathrm{PC}}$ conditional on any planted set $S$,
\begin{equation}\label{eq_sparse_cut_concentration}
\mathbb P\left(
\|\widehat\mu-\mu\|_{\square,k}
>
C\sqrt{k\log\left(\frac{ep}{k}\right)}
\right)
\le
2\exp\left(
-ck\log\left(\frac{ep}{k}\right)
\right).
\end{equation}
\end{lemma}

\subsubsection{Proof of Lemma~\ref{lem_sparse_spectral_concentration}}\label{sec_proof_lem_sparse_spectral_concentration}

\begin{proof}
Let $W=\widehat\mu-\mu$.
Under the null, the entries $W_{ij}$ for $i<j$ are independent random variables taking the values $1/2$ and $-1/2$ with equal probability.
Under the alternative, conditional on $S$, the entries within $S\times S$ are zero, while the remaining upper-triangular entries are independent, centered, and bounded in absolute value by $1/2$.

For any fixed unit vector $u$ supported on at most $k$ coordinates,
\[
u^\top W u
=
2\sum_{i<j}u_i u_j W_{ij}.
\]
Since
\[
\sum_{i<j}u_i^2u_j^2
=
\frac12\left(1-\sum_i u_i^4\right)
\le
\frac12,
\]
Hoeffding's inequality gives a universal constant $c_0>0$ such that
\begin{equation}\label{eq_fixed_sparse_quadratic_tail}
\mathbb P\left(
|u^\top W u|>t
\right)
\le
2\exp(-c_0t^2).
\end{equation}

For each nonempty set $T\subseteq[p]$ with $|T|\le k$, let $\mathcal N_T$ be a $1/4$-net of the unit sphere supported on $T$.
The net can be chosen so that $|\mathcal N_T|\le9^{|T|}$.
The standard net bound for the operator norm of a symmetric matrix gives
\[
\|W_{T,T}\|_{\mathrm{op}}
\le
2\max_{u\in\mathcal N_T}|u^\top W u|.
\]
Moreover, for a universal constant $C_0>0$,
\[
\sum_{j=1}^k\binom pj9^j
\le
\exp\left(
C_0k\log\left(\frac{ep}{k}\right)
\right).
\]
Combining this counting bound with \eqref{eq_fixed_sparse_quadratic_tail} and taking a union bound yields
\[
\mathbb P\left(
\|W\|_{\mathrm{sp},k}>2t
\right)
\le
2\exp\left(
C_0k\log\left(\frac{ep}{k}\right)-c_0t^2
\right).
\]
Choosing $t=C_1\sqrt{k\log(ep/k)}$ for a sufficiently large universal constant $C_1$ proves \eqref{eq_sparse_spectral_concentration}.
\end{proof}

\subsubsection{Proof of Lemma~\ref{lem_sparse_cut_concentration}}\label{sec_proof_lem_sparse_cut_concentration}

\begin{proof}
Let $W=\widehat\mu-\mu$. Under the null, the entries $W_{ij}$ for $i<j$ are independent Rademacher variables scaled by $1/2$. Under the alternative, conditional on $S$, entries within $S\times S$ are zero and the remaining upper-triangular entries are independent, centered, and bounded in absolute value by $1/2$.

Fix nonempty $R,C\subseteq[p]$, and write $r=|R|$ and $c=|C|$. Since $W$ is symmetric with zero diagonal,
\[
\frac{\mathbf 1_R^\top W\mathbf 1_C}{\sqrt{rc}}
=
\sum_{i<j}
a_{ij}W_{ij},
\qquad
a_{ij}
=
\frac{
\mathbf 1_R(i)\mathbf 1_C(j)
+
\mathbf 1_R(j)\mathbf 1_C(i)
}{\sqrt{rc}}.
\]
Using $(x+y)^2\le2x^2+2y^2$,
\[
\sum_{i<j}a_{ij}^2
\le
\frac2{rc}
\sum_{i<j}
\left[
\mathbf 1_R(i)\mathbf 1_C(j)
+
\mathbf 1_R(j)\mathbf 1_C(i)
\right]
\le
2.
\]
Hoeffding's inequality therefore gives universal $c_0>0$ such that
\begin{equation}\label{eq_fixed_sparse_cut_tail}
\mathbb P\left(
\frac{|\mathbf 1_R^\top W\mathbf 1_C|}{\sqrt{rc}}>t
\right)
\le
2\exp(-c_0t^2).
\end{equation}

The number of ordered pairs $(R,C)$ with $1\le|R|,|C|\le k$ is at most
\[
\left(
\sum_{j=1}^k\binom pj
\right)^2
\le
\exp\left(
C_0k\log\left(\frac{ep}{k}\right)
\right)
\]
for a universal constant $C_0$. A union bound in \eqref{eq_fixed_sparse_cut_tail} yields
\[
\mathbb P\left(
\|W\|_{\square,k}>t
\right)
\le
2\exp\left(
C_0k\log\left(\frac{ep}{k}\right)-c_0t^2
\right).
\]
Taking $t=C_1\sqrt{k\log(ep/k)}$ for a sufficiently large universal constant $C_1$ proves \eqref{eq_sparse_cut_concentration}.
\end{proof}

\subsection{Proof of Theorem~\ref{thm_sparse_spectral_planted_clique}}\label{sec_proof_thm_sparse_spectral_planted_clique}

\begin{proof}
If $f_{p,k}$ is randomized, repeat it independently $C\log p$ times and take the median of its outputs.
A Chernoff bound shows that the amplified algorithm satisfies \eqref{eq_sparse_spectral_approximation} with probability $1-o(1)$ for every input, without changing its approximation ratios or its polynomial running time.

Suppose, toward a contradiction, that
\[
\gamma_{f,p}
<
\frac{s_p}{4r_p},
\]
for all sufficiently large $p$, where $s_p=(k-1)/2$ and $r_p$ is given by \eqref{eq_planted_clique_estimation_radius}.

Since $k$ is public and $\tau_{p,k}$ is uniformly polynomial-time computable from $(p,k)$, the following is a uniform randomized polynomial-time test:
\begin{equation}\label{eq_planted_clique_induced_test}
\phi_p=I_{\left\{f_{p,k}(\widehat\mu)>\tau_{p,k}\right\}}.
\end{equation}

Under $\mathbb Q_p^{\mathrm{PC}}$, Lemma~\ref{lem_sparse_spectral_concentration} gives $\|\widehat\mu\|_{\mathrm{sp},k}\le r_p$ with probability tending to one.
On this event and on the success event of the amplified approximation algorithm,
\[
f_{p,k}(\widehat\mu)
\le
\zeta_{f,p}r_p
=
\frac{\gamma_{f,p}r_p}{\rho_{f,p}}
<
\frac{s_p}{4\rho_{f,p}}
<
\frac{s_p}{2\rho_{f,p}}.
\]
Therefore, $\mathbb Q_p^{\mathrm{PC}}(\phi_p=1)\to0$.

Under $\mathbb P_p^{\mathrm{PC}}$, Lemma~\ref{lem_sparse_spectral_concentration} and the triangle inequality give
\[
\|\widehat\mu\|_{\mathrm{sp},k}
\ge
\|\mu_S\|_{\mathrm{sp},k}
-
\|\widehat\mu-\mu_S\|_{\mathrm{sp},k}
\ge
s_p-r_p
\]
with probability tending to one.
Since $\gamma_{f,p}\ge1$, the assumed inequality implies $r_p<s_p/4$ for all sufficiently large $p$.
Consequently, on the preceding concentration event and on the success event of the amplified approximation algorithm,
\[
f_{p,k}(\widehat\mu)
\ge
\frac{\|\widehat\mu\|_{\mathrm{sp},k}}{\rho_{f,p}}
>
\frac{3s_p}{4\rho_{f,p}}
>
\frac{s_p}{2\rho_{f,p}}.
\]
Therefore, $\mathbb P_p^{\mathrm{PC}}(\phi_p=0)\to0$.

The test \eqref{eq_planted_clique_induced_test} strongly distinguishes $\mathbb P_p^{\mathrm{PC}}$ from $\mathbb Q_p^{\mathrm{PC}}$, contradicting Assumption~\ref{assump_planted_clique}. Therefore,
\[
\limsup_{p\to\infty}
\frac{\gamma_{f,p}r_p}{s_p}
\ge
\frac14.
\]
Combining this with \eqref{eq_planted_clique_reverse_gap} proves \eqref{eq_sparse_spectral_lower_bound}.
\end{proof}

\subsection{Proof of Theorem~\ref{thm_sparse_cut_planted_clique}}\label{sec_proof_thm_sparse_cut_planted_clique}

\begin{proof}
If $g_{p,k}$ is randomized, repeat it independently $C\log p$ times and take the median of its outputs. A Chernoff bound makes the amplified approximation guarantee hold with probability $1-o(1)$ for every input, without changing the approximation factors or the polynomial running time.

Suppose, toward a contradiction, that
\[
\gamma_{g,p}
<
\frac{s_p}{4r_p}
\]
for all sufficiently large $p$, where $s_p=(k-1)/2$ and $r_p$ is given in \eqref{eq_planted_clique_estimation_radius}. Since $k$ is public and $\tau_{p,k}^{\square}=s_p/(2\rho_{g,p})$ is uniformly polynomial-time computable, define the uniform randomized polynomial-time test
\[
\phi_p
=
I_{\{g_{p,k}(\widehat\mu)>\tau_{p,k}^{\square}\}}.
\]

Under $\mathbb Q_p^{\mathrm{PC}}$, Lemma~\ref{lem_sparse_cut_concentration} gives $\|\widehat\mu\|_{\square,k}\le r_p$ with probability tending to one. On this event and on the success event of the amplified approximation algorithm,
\[
g_{p,k}(\widehat\mu)
\le
\zeta_{g,p}r_p
=
\frac{\gamma_{g,p}r_p}{\rho_{g,p}}
<
\frac{s_p}{4\rho_{g,p}}
<
\tau_{p,k}^{\square}.
\]
Thus $\mathbb Q_p^{\mathrm{PC}}(\phi_p=1)\to0$.

Under $\mathbb P_p^{\mathrm{PC}}$, Lemma~\ref{lem_sparse_cut_concentration}, the triangle inequality, and \eqref{eq_planted_clique_cut_signal} give
\[
\|\widehat\mu\|_{\square,k}
\ge
\|\mu_S\|_{\square,k}
-
\|\widehat\mu-\mu_S\|_{\square,k}
\ge
s_p-r_p
\]
with probability tending to one. Since $\gamma_{g,p}\ge1$, the assumed inequality implies $r_p<s_p/4$ for all sufficiently large $p$. Hence, on the preceding concentration event and on the success event of the amplified algorithm,
\[
g_{p,k}(\widehat\mu)
\ge
\frac{\|\widehat\mu\|_{\square,k}}{\rho_{g,p}}
>
\frac{3s_p}{4\rho_{g,p}}
>
\tau_{p,k}^{\square}.
\]
Thus $\mathbb P_p^{\mathrm{PC}}(\phi_p=0)\to0$, contradicting Assumption~\ref{assump_planted_clique}. Therefore
\[
\limsup_{p\to\infty}
\frac{\gamma_{g,p}r_p}{s_p}
\ge
\frac14.
\]
Combining this with \eqref{eq_planted_clique_reverse_gap} proves \eqref{eq_sparse_cut_lower_bound}.
\end{proof}

\subsection{Proofs for the tensor cut-norm reverse gap}\label{sec_proofs_tensor_cut_reverse_gap}

\begin{proof}[Proof of Lemma~\ref{lem_tensor_cut_norm}]
Absolute homogeneity and the triangle inequality are immediate. To verify definiteness, fix arbitrary indices $i_1,\ldots,i_d$ and take $I_j=\{i_j\}$ for every $j$. If $\|T\|_{\square,d}=0$, the corresponding rectangle sum $T_{i_1,\ldots,i_d}$ is zero. Hence every tensor entry is zero.
\end{proof}

\begin{proof}[Proof of Lemma~\ref{lem_tensor_cut_signal}]
For each mode $j$, one of the sets $\{i:u_i^{(j)}=p^{-1/2}\}$ and $\{i:u_i^{(j)}=-p^{-1/2}\}$ has cardinality at least $p/2$. Choose $I_j$ to be such a set. Then $\left|\sum_{i\in I_j}u_i^{(j)}\right|\ge\sqrt p/2$.
The rectangle sum factorizes across modes, giving
\[
p^{-d/2}\lambda_p
\prod_{j=1}^d
\left|
\sum_{i\in I_j}u_i^{(j)}
\right|
\ge
2^{-d}\lambda_p.
\]
\end{proof}

\begin{proof}[Proof of Lemma~\ref{lem_tensor_cut_noise}]
For a fixed rectangle $(I_1,\ldots,I_d)$, its normalized Gaussian sum has variance
\[
p^{-d}\prod_{j=1}^d|I_j|
\le
1.
\]
There are at most $2^{dp}$ rectangles. Therefore, for every $t>0$,
\[
\mathbb P\left(
\|Z\|_{\square,d}>t
\right)
\le
2^{dp+1}\exp(-t^2/2).
\]
Taking $t=C_d\sqrt p$ with $C_d$ sufficiently large proves the claim.
\end{proof}

\begin{proof}[Proof of Proposition~\ref{prop_tensor_cut_ldlr}]
For an additive Gaussian mixture, the Hermite expansion gives
\[
\left\|(L_p^{\mathrm{TC}})_{\le D}\right\|_{L^2(\mathbb Q_p^{\mathrm{TC}})}^2
=
\sum_{m=0}^D
\frac1{m!}
\E_{\mu,\mu'}
\langle\mu,\mu'\rangle^m,
\]
where $\mu$ and $\mu'$ are independent draws from the spike prior. Write $v^{(j)}$ for the $j$th factor of $\mu'$. Since
\[
\langle\mu,\mu'\rangle
=
\lambda_p^2
\prod_{j=1}^d
\langle u^{(j)},v^{(j)}\rangle
\]
and the $d$ overlaps are independent copies of $R$, \eqref{eq_tensor_cut_ldlr_exact} follows.

The variable $R$ is centered and $p^{-1/2}$-sub-Gaussian, so
\[
\left(\E|R|^m\right)^{1/m}
\le
C\sqrt{\frac mp}.
\]
Using $m!\ge(m/e)^m$ gives \eqref{eq_tensor_cut_ldlr_term}. Under \eqref{eq_tensor_cut_lambda}, for $m\le c(\log p)^2$ the base in \eqref{eq_tensor_cut_ldlr_term} is at most
\[
\frac{C_{d,c}m^{d/2-1}}{(\log p)^{2B}}
\le
C_{d,c}(\log p)^{d-2-2B}
=
o(1).
\]
Summing the resulting geometric bound proves \eqref{eq_tensor_cut_ldlr_bounded}.
\end{proof}

\subsection{Proof of Theorem~\ref{thm_tensor_cut_lower_bound}}\label{sec_proof_thm_tensor_cut_lower_bound}

\begin{proof}
Take $\widehat\mu=Y$, $s_p=2^{-d}\lambda_p$, and $r_p=C_d\sqrt p$. Proposition~\ref{prop_tensor_cut_ldlr} and Conjecture~\ref{conj_tensor_cut_low_degree} give the strong-detection hardness condition. Lemmas~\ref{lem_tensor_cut_signal} and~\ref{lem_tensor_cut_noise} give the separation and estimation-radius conditions. Because $f_p$ and the ratios $s_p/(2\rho_{f,p})=2^{-d-1}\lambda_p/\rho_{f,p}$ are uniformly polynomial-time computable, threshold closure holds. Theorem~\ref{thm_framework} therefore gives
\[
\limsup_{p\to\infty}
\frac{\gamma_{f,p}r_p}{s_p}
\ge
\frac14.
\]
Substituting \eqref{eq_tensor_cut_reverse_gap} proves \eqref{eq_tensor_cut_lower_bound}.
\end{proof}
\end{document}